\documentclass[notitlepage,a4,10.5pt]{article}

\usepackage[colorlinks=true,citecolor=blue,linkcolor=blue]{hyperref}

\usepackage[left=2.5cm,right=2.5cm,top=3cm,bottom=3cm]{geometry} 

\usepackage{tikz}
\usetikzlibrary{arrows.meta,positioning,calc,shapes,fit}
\usepackage{subcaption}

\usepackage{tabularx}
\usepackage{array}
\newcolumntype{Y}{>{\raggedright\arraybackslash}X}
\usepackage{booktabs}  
\usepackage{multirow}  

\usepackage{graphicx}

\usepackage{amsmath}%
\numberwithin{equation}{section}

\usepackage{amsfonts}%
\usepackage{amssymb}%
\usepackage{mathrsfs}
\usepackage{amsthm}
\usepackage{float}

\usepackage{authblk}
\usepackage{enumitem}

\usepackage{xcolor}
\usepackage{appendix}

\newcommand{\ml}{\mathcal}
\newcommand{\cp}{\times}

\newcommand{\bol}{\boldsymbol}

\newcommand{\abs}[1]{\left\lvert{#1}\right\rvert}
\newcommand{\w}{\wedge}
\newcommand{\lr}[1]{\left({#1}\right)}
\newcommand{\lrs}[1]{\left[{#1}\right]}
\newcommand{\lrc}[1]{\left\{{#1}\right\}}

\newcommand{\mf}{\mathfrak}
\newcommand{\p}{\partial}

\newtheorem{thm}{\textit{Theorem}}[section]

\theoremstyle{definition}
\newtheorem{mydef}[thm]{Definition}

\theoremstyle{definition}

\theoremstyle{remark}
\newtheorem{remark}[thm]{\textbf{Remark}}

\newtheorem{prop}[thm]{\textbf{Proposition}}

\newtheorem{lemma}[thm]{\textbf{Lemma}}
\newtheorem{q}[thm]{\textbf{Question}}

\newcommand{\eq}[1]{\begin{equation}\begin{split}{#1}\end{split}\end{equation}}
\newcommand{\sys}[2]{\begin{subequations}\begin{align}{#1}\end{align}\label{#2}\end{subequations}}

\definecolor{myyellow}{RGB}{255,255,180}

\begin{document}

\title{
Admissible Invariant-Torus Foliations for Steady Euler Flows
}
\author{Naoki Sato\thanks{Corresponding author} {}\thanks{National Institute for Fusion Science,  322-6 Oroshi-cho Toki-city, Gifu 509-5292, Japan,\ Email: \href{sato.naoki@nifs.ac.jp}{sato.naoki@nifs.ac.jp}
}  
{}\thanks{Graduate School of Frontier Sciences, The University of Tokyo,  
Kashiwa, Chiba 277-8561, Japan}
\qquad Ken Abe \thanks{Department of Mathematics, Graduate School of Science,  Osaka Metropolitan University 3-3-138, Sugimoto, Sumiyoshi-ku, Osaka 558-8585, Japan, Email: \href{kabe@omu.ac.jp}{kabe@omu.ac.jp}}}
 
\date{\today}
\setcounter{Maxaffil}{0}
\renewcommand\Affilfont{\itshape\small}

    \maketitle

\begin{abstract}
In 1965, V.~I.~Arnold established a structure theorem guaranteeing the existence of a foliation by invariant surfaces for general three-dimensional steady Euler flows with non-constant pressure. In this paper, we investigate what foliation structures can arise in steady Euler flows.

We consider a toroidal domain foliated by the level sets of a flux function $\Psi$, and prove that every $C^{1}$ steady Euler flow $(\bol{u},p)$ satisfying the assumptions $\iota_{\bol{u}}d\Psi=0$ and $p=p(\Psi)$ admits the tangential flow representation
\[
\bol{u}=c_1(\Psi)\bol{\xi}^{1}+c_2(\Psi)\bol{\xi}^{2},
\]
for some lifted solenoidal vector fields $\bol{\xi}^{1}$ and $\bol{\xi}^{2}$ associated with a natural basis of weighted harmonic one-forms on the toroidal leaves. Moreover, the flux function $\Psi$ satisfies a single scalar equation, referred to as the normal flux equation. These characterizations reveal the general foliation structure of steady Euler flows, with the Clebsch representation and the Grad--Shafranov equation recovered as the axisymmetric special case.
\end{abstract}

\if{
\begin{abstract}
{
We study admissible invariant-torus foliations of steady Euler flows in a smooth hollow toroidal domain. More precisely, given a smooth foliation by embedded tori, we ask when it can support a steady Euler flow that is tangent to the leaves and whose Bernoulli function is constant on each leaf. Under these assumptions, we show that the steady Euler system separates into a leafwise Hodge-theoretic constraint and a transverse scalar compatibility condition. The leafwise constraint forces the tangential velocity $1$-form to be weighted harmonic on each torus. This harmonic decomposition generalizes the Clebsch representation appearing in the axisymmetric case. On each torus, prescribing the two periods selects a unique weighted harmonic representative through a leafwise elliptic problem. Substitution into the transverse condition yields a closed scalar equation for the foliation, which is pointwise in the leaf label but generally nonlocal along each torus. In the presence of axial symmetry, this foliation-admissibility equation reduces to the classical Grad--Shafranov equation. 
\end{abstract}
}\fi


\tableofcontents

\section{Introduction}
\subsection{Foliations of steady Euler flows
}\label{subsec:clGS}

This paper is concerned with the steady Euler equations
\eq{
\iota_{\bol{u}}du=-dp,
\qquad
d\iota_{\bol{u}}dV=0,
\label{SE}
}
in a smooth bounded domain $\Omega\subset\mathbb{R}^3$.
Here $\bol{u}\lr{\bol{x}}$ denotes the velocity vector field, $u=\bol{u}^{\flat}$ its associated one-form, $p\lr{\bol{x}}$ the Bernoulli function, $\bol{x}=\lr{x^1,x^2,x^3}\in\mathbb{R}^3$, and $dV$ the Euclidean volume form on $\mathbb{R}^3$. The one-form equations \eqref{SE} are the dual equations to the vector form  
\eq{
(\nabla\times  \bol{u})\times \bol{u}=-\nabla p,
\qquad
\nabla \cdot \bol{u}=0.
\label{SEV}
}

Classical constructions of solutions to \eqref{SE} are often based on symmetry assumptions. In a smooth axisymmetric toroidal domain $\Omega$ with cylindrical
coordinates $(r,\phi,z)$, an axisymmetric solenoidal vector field
admits the Clebsch representation
\begin{align}
\bol{u}
 = \nabla\Psi(r,z)\times\nabla\phi
   +\alpha(r,z)\nabla\phi,
\label{Clebsch}
\end{align}
where $\Psi$ is a flux function and $\alpha$ is a scalar function;
see, e.g., \cite{Yoshida09}. The function $\Psi$ is a first integral of
$\bol{u}$, i.e., $\iota_{\bol{u}}d\Psi=0$, and hence $\bol{u}$ is tangent to each flux surface
$\Sigma_{\Psi_0}=\{\Psi=\Psi_0\}$. If $\alpha$ and $p$ depend locally only on
$\Psi$, the momentum equation in \eqref{SE} reduces to the
Grad--Shafranov equation \cite{Grad58,Shafranov1958}:
\begin{align}
\Delta^*\Psi+\alpha(\Psi)\alpha'(\Psi)
 =r^2p'(\Psi),
\label{GSeq0}
\end{align}
where $\Delta^*=\partial_r^2+\partial_z^2-r^{-1}\partial_r$ is the Grad--Shafranov operator. Thus, under axisymmetry, the steady
Euler equations decompose into the Clebsch representation
\eqref{Clebsch} and the Grad--Shafranov equation \eqref{GSeq0}.
Conversely, for globally  prescribed functions $\alpha=\alpha(\Psi)$ and
$p=p(\Psi)$, a solution of \eqref{GSeq0} yields an axisymmetric
solution of \eqref{SE} through \eqref{Clebsch}
\cite{Gavrilov2019,ConstantinLaVicol2019,Dominguez2021}; see also \cite{Abe2022}.

Axisymmetric solutions to \eqref{SE} whose regular level sets of
$\Psi$ form nested tori admit invariant-torus foliations. The velocity
field $\bol{u}$ is tangent to, and winds on, each torus
$\Sigma_{\Psi_0}$; see Figure~\ref{fig1}. Moreover,
the restriction of the velocity one-form to each torus is weighted
harmonic. More precisely,
\begin{align}
\iota_{\nabla\times\bol{u}}d\Psi=0,
\qquad
{d_{\Sigma_{\Psi_0}}
\left(
\iota_{\bol{u}}d\Sigma_{\Psi_0}
\right)=0},
\label{eq:weighted-harmonic-axisym}
\end{align}
where the leafwise measure $d\Sigma_{\Psi_0}$ is determined by $dV=d\Psi\wedge d\Sigma_{\Psi_0}$ {and $d_{\Sigma_{\Psi_0}}$ denotes the exterior derivative on the leaf $\Sigma_{\Psi_0}$}.

Even without symmetry assumptions, Arnold's classical structure
theorem \cite{Arnold65, Arnold1966}, \cite[II, Theorem 1.2]{ArnoldKh} guarantees, in the analytic
setting, that a general steady Euler flow with nonconstant pressure
admits a decomposition into regions foliated by invariant tori or
invariant cylinders; cf. \cite[Problem 33]{KMS23}.

\begin{thm}[Arnold's structure theorem] \label{thm:Arnold}
Assume that the region $\Omega \subset\mathbb{R}^3$ is bounded by a compact analytic surface, and that the field of velocities is 
analytic and not everywhere collinear with its curl. Then the region of the flow
can be partitioned by an analytic submanifold into a finite number of cells, in
each of which the flow is constructed in a standard way. Namely, the cells are
of two types: those fibered into tori invariant under the flow and those fibered
into surfaces invariant under the flow, diffeomorphic to the annulus $\mathbb{R}\times S^1$. On 
each of these tori the flow lines are either all closed or all
dense, and on each annulus all flow lines are closed.
\end{thm}

Arnold's theorem guarantees the existence of invariant-surface
foliations for analytic steady Euler flows with nonconstant pressure.
It does not, however, provide an explicit characterization of the
foliations compatible with the Euler equations, nor an equation that
reconstructs the flow from a prescribed foliation. By contrast, in
the axisymmetric setting, the Clebsch representation
\eqref{Clebsch} and the Grad--Shafranov equation \eqref{GSeq0}
give an explicit description of both the velocity field and its
invariant-torus foliation.

\begin{figure}[h]
  \centering  \includegraphics[width=0.40\textwidth]{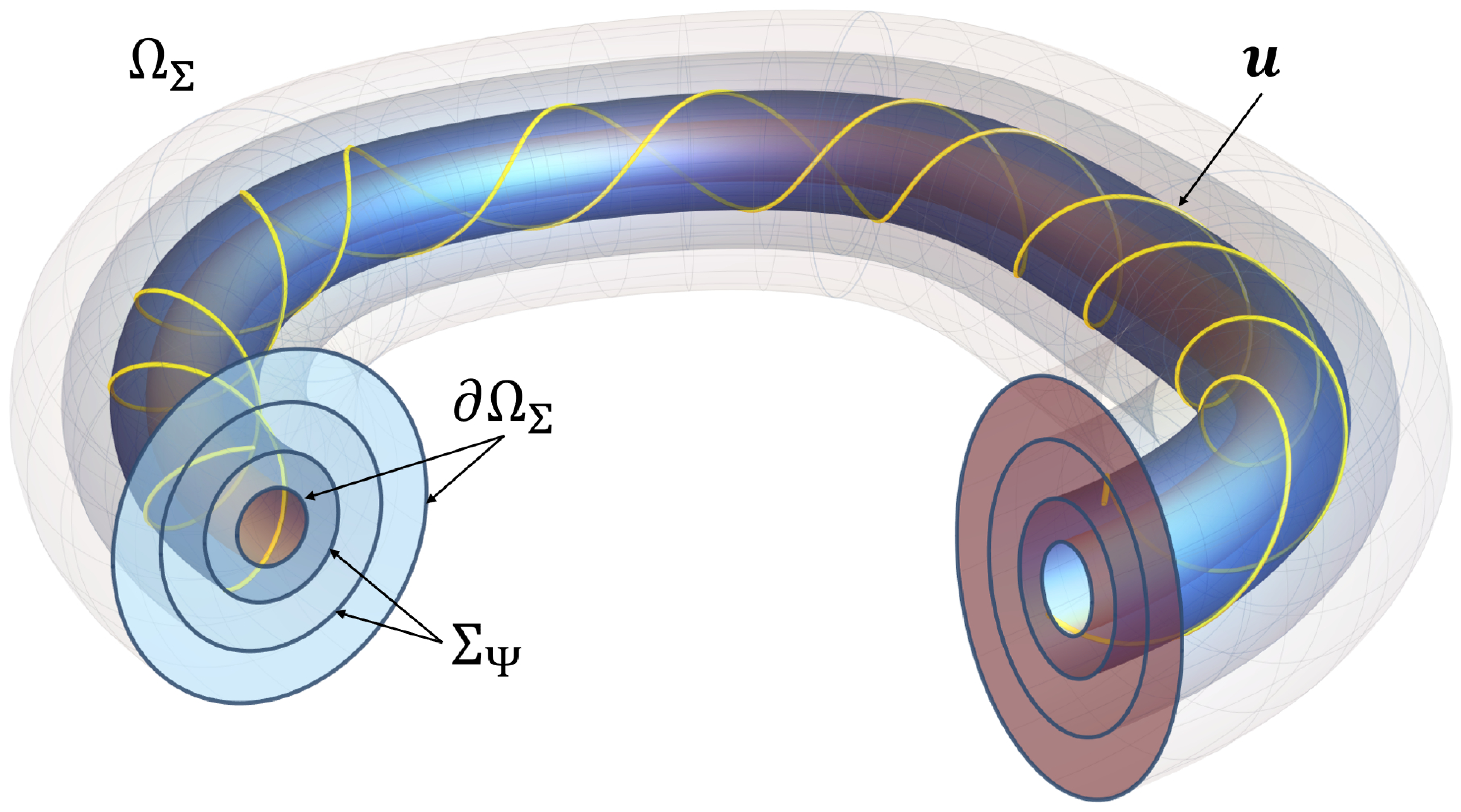}
  \caption{A hollow toroidal domain
  $\Omega_{\Sigma}$ foliated by nested invariant tori
  $\Sigma_{\Psi_0}=\{\Psi=\Psi_0\}$. The foliation is cut
  to display a section of the hollow toroidal structure. 
  }
  \label{fig1}
\end{figure}

In plasma physics, solutions of \eqref{SE} describe
magnetohydrostatic (MHS) equilibria. H. Grad conjectured that every MHS
equilibrium with nonconstant pressure must possess a continuous
Euclidean symmetry \cite{Grad67,Grad85}; see also
\cite[Conjecture~1]{ConstantinDrivasGinsberg21}. The underlying intuition is that the existence of a smooth foliation
by magnetic surfaces severely constrains the magnetic field and is
therefore expected to require continuous symmetry.

From the viewpoint of Theorem~\ref{thm:Arnold}, Grad's conjecture can
be interpreted as a geometric rigidity statement: invariant
foliations arising in steady Euler flows with nonconstant pressure
should be induced by continuous Euclidean symmetries. The background
leading to this viewpoint is summarized in Figure~\ref{f2} (I).
This leads to the following question, which, to the best of our
knowledge, has not previously been formulated explicitly.

\begin{q}[Admissible foliations]\label{q0}
Which foliations can arise in steady Euler flows \eqref{SE} with non-constant pressure?
\end{q}

The existence of nonsymmetric steady Euler flows admitting
invariant-torus foliations has been discussed by Weitzner
\cite{Weitzner14}. More recently, $m$-fold symmetric steady solutions
whose invariant cylinders are rational surfaces were constructed in
\cite{Drivas25}; cf.~\cite{Lortz70}. In the H\"older-continuous category,
steady Euler flows with streamlines of arbitrary topology have also
been constructed by convex integration \cite{EncisoPenafielTomasPeraltaSalas2025TopologyPreserving}. On the other
hand, Peralta-Salas and Slobodeanu \cite{PeraltaSalas2026} proved that every
localizable steady Euler flow, that is, every steady Euler flow for
which $|\bol{u}|^2$ is a first integral, necessarily possesses a
continuous symmetry; see also \cite{Gavrilov2019, ConstantinLaVicol2019,Dominguez2021}.

\subsection{Geometric setup}

Figure~\ref{f2} (II) illustrates our approach. In this paper, we address Question~\ref{q0} within the class of steady
Euler flows admitting an invariant-torus foliation in a toroidal
domain. Motivated by Arnold's structure theorem, we consider steady Euler
flows whose level sets of the pressure (the Bernoulli surfaces) form a family of nested tori. More precisely, we
consider the following geometric structure:
\begin{itemize}
\item the velocity field $\bol{u}$ is tangent to the level sets of
      the pressure $p$, namely,
      \[
      \iota_{\bol{u}}dp=0;
      \]
\item the regular level sets $\{p=p_0\}$ form a family of nested tori.
\end{itemize}
We encode this structure by a flux function $\Psi$ whose regular
level sets are nested tori. We then regard the pressure as a function
of $\Psi$. Accordingly, we consider steady Euler flows satisfying
\begin{align}
\iota_{\bol{u}}d\Psi&=0,
\label{eq:A1}\\
p&=p(\Psi).
\label{eq:A2}
\end{align}

\begin{figure}[h]
\centering

\usetikzlibrary{positioning}

\begin{tikzpicture}[
box/.style={
    draw,
    rounded corners,
    align=left,
    text width=0.40\linewidth,   
    inner sep=2mm                
},
node distance=6mm and 1.5cm
]

\node[box] (L1) at (-4.3,0) {$\bigstar$ Axisymmetric steady states to \eqref{SE} admit invariant-torus foliations with 
\begin{itemize}
\item Clebsch representation \eqref{Clebsch}
\item Grad--Shafranov equation \eqref{GSeq0}
\end{itemize}};

\node[box] (L2) at (-4.3,-2.0) {(1) General steady states to \eqref{SE}
};

\node (A2) at ($(L2.south)+(-3.0,-0.8)$) {\LARGE $\Downarrow$};

\node[
    anchor=west,
    font=\small
] at ($(A2.east)+(2mm,0)$)
{{Arnold's structure theorem (Theorem \ref{thm:Arnold})}};

\node[box] (L3) at (-4.3,-4.5) {(2) Existence of foliations by invariant surfaces (tori/cylinders)};

\node (AL3) at ($(L3.south)+(-3.0,-0.8)$) {\LARGE $\Downarrow$};

\node[
    anchor=west,
    font=\small
] at ($(AL3.east)+(2mm,0)$)
{Grad's conjecture on symmetry};

\node[box] (L4) at (-4.3,-7.2) {(3) Which foliations can arise? 
(Question \ref{q0})};

\node[
    fit=(L1)(L2)(L3)(L4),
    inner sep=6mm,
    label={[font=\bfseries]above:
        (I) The background leading to Question \ref{q0}}
] {};

\node[box] (R1) at (4.3,0) {(1) General steady states to \eqref{SE} with invariant-torus foliations satisfying 
\begin{itemize}
\item Flux condition \eqref{eq:A1} 
\item Pressure condition \eqref{eq:A2}
\end{itemize} 
};

\node (B1) at ($(R1.south)+(-3.0,-0.9)$) {\Large $\Downarrow$};

\node[
    anchor=west,
    font=\small
] at ($(B1.east)+(2mm,0)$)
{Weighted Hodge-theory};

\node[box] (R2) at (4.3,-4.3) {(2) Foliation structure (Theorem \ref{thm:lnGS}; the main result of this paper)
\begin{itemize}
\item Tangential flow representation \eqref{UA}
\item Normal flux equation \eqref{lnGS}
\end{itemize}
};

\node (B2) at ($(R1.south)+(-3.0,-5.1)$) {\Large $\Downarrow$};

\node[
    anchor=west,
    font=\small
] at ($(B2.east)+(2mm,-0.3mm)$)
{Symmetry};

\node[box] (R3) at (4.3,-7.7) {$\bigstar$ Axisymmetric steady states to \eqref{SE} (Theorem \ref{t:axisymmetric})};

\node[
    fit=(R1)(R2)(R3),
    inner sep=6mm,
    label={[font=\bfseries]above:
        (II) This paper's approach to Question \ref{q0}}
] {};
\end{tikzpicture}
\caption{The background leading to Question \ref{q0} and this paper's approach}\label{f2}
\end{figure}
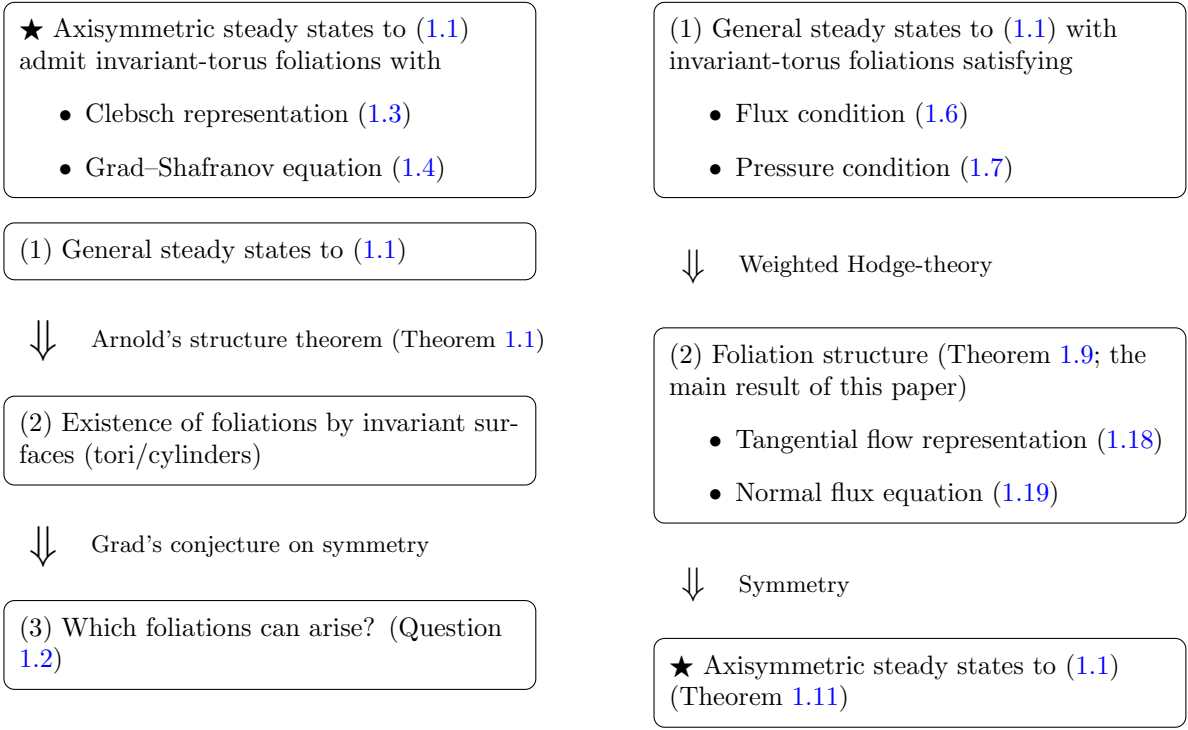

To avoid unnecessary technical complications, we assume that
the toroidal domain has a hollow core; see Figure \ref{fig1}.

\begin{mydef}[Foliated toroidal domain]\label{def:marked-foliation}
Let $\Omega_{\Sigma}\subset\mathbb{R}^3$ be a smooth bounded domain. We say
that $\Omega_{\Sigma}$ is a (hollow) foliated toroidal domain if there
exists a smooth submersion $\Psi:\Omega_{\Sigma}\to I\subset \mathbb{R}$ such
that $\nabla\Psi\neq\bol{0}$ and $\Omega_{\Sigma}=\cup_{\Psi_0\in I}\Sigma_{\Psi_0}$ with smooth tori
\eq{
\Sigma_{\Psi_0}
=
\{\bol{x}\in\Omega_{\Sigma}:\Psi(\bol{x})=\Psi_0\}
\simeq \mathbb{T}^2 .  \label{eq:ident}
}
We call $\Psi$ a flux function. We denote the inclusion map by  $i:\Sigma_{\Psi_0}\hookrightarrow\mathbb{R}^3$. 
\end{mydef}

The flux function $\Psi$ is not prescribed a priori; it is part of
the unknown solution of \eqref{SE}. For a given candidate foliation,
we construct angle coordinates $(\theta^1,\theta^2)$ on each leaf
$\Sigma_\Psi$ and the corresponding adapted coordinates
$(\Psi,\theta^1,\theta^2)$ on $\Omega_\Sigma$.

\begin{mydef}[Leafwise coordinates] \label{n:lg}

Let $T$ be a smooth vector field $T$ such that 
\begin{align}
\iota_Td\Psi=1.  \label{eq:tran}
\end{align}
On each torus $\Sigma_{\Psi_0}$ and $\Psi_0\in I$, we set the following.
\begin{enumerate}[label=\textup{(\roman*)}, leftmargin=*, itemsep=0.6em]
\item (Coordinates).
We choose normalized angle coordinates $\theta^j\in\mathbb{R}/\mathbb{Z}$ for $j=1,2$, and denote the normalized linearly independent period covectors by
$P^i=P_j^i d\theta^j$ for $i=1,2$, with some constants $P_j^i$ (The simplest choice is $P^{i}_{j}=\delta_{i,j}$ for which $P^{i}=d\theta^{i}$).

\item (Metric and measure).
We denote the
induced metric and the leafwise measure on $\Sigma_{_{\Psi_0}}$ by
\eq{
g_{_{\Psi_0}}=h_{ij}(\Psi_0,\theta^{1},\theta^{2})d\theta^i d\theta^j,
\qquad
d\Sigma_{\Psi_0}=\iota_TdV|_{\Sigma_{\Psi_0}}
=M(\Psi_0,\theta^{1},\theta^{2})d\theta^1\wedge d\theta^2. \label{eq:MM}
}
The measure $d\Sigma_{\Psi_0}$ is independent of the choice of such a vector field $T$.

\item (Elliptic operator).
We define the leafwise elliptic operator on $\Sigma_{\Psi_0}$ by
\eq{
L_{\Psi_0}
=
-\p_i\left(Q^{ij}\p_j\right),
\qquad
Q^{ij}=M h^{ij},  \label{eq:Ell}
}
where $\p_j=\p/\p\theta^j$ and $(h^{ij})=(h_{ij})^{-1}$. The operator \(L=L_{\Psi_0}\) is invertible on spaces of average zero functions on $\Sigma_{\Psi_0}$.
\end{enumerate}
\end{mydef}

\begin{mydef}[Weighted Hodge-star operator]
Let $dV_{g_{\Psi_0}}$ denote the volume form on $\Sigma_{\Psi_0}$ induced from the metric $g_{\Psi_0}$. We set the weighted Hodge-star operator by 
\eq{
\tilde{\ast}=f\ \ast:\Lambda^k(\Sigma_{\Psi_0})\to\Lambda^{2-k}(\Sigma_{\Psi_0}),\quad f=\frac{d\Sigma_{\Psi_0}}{dV_{g_{\Psi_0}}},
\label{eq:wHodgestar}}
where $\Lambda^k(\Sigma_{\Psi_0})$ denotes the space of $k$-forms on $\Sigma_{\Psi_0}$.
\end{mydef}

\begin{mydef}[Adapted coordinates]\label{def:adapted}
We set an adapted coordinate system on a foliated toroidal domain $\Omega_\Sigma$ by 
\begin{equation}
(\Psi,\theta^1,\theta^2)\in I\times (\mathbb{R}/\mathbb{Z})^{2}.
\label{eq:adopted}
\end{equation}
The Euclidean metric in this coordinates is expressed as  
\begin{equation}
g_{\mathrm{Euc}}
=
g_{\Psi\Psi}\,d\Psi^2
+
2g_{\Psi\theta^i}\,d\Psi\,d\theta^i
+
h_{ij}\,d\theta^i d\theta^j.  \label{eq:metric}
\end{equation}
\end{mydef}

The leafwise measure, metric, and elliptic operator depend smoothly
on the leaf parameter $\Psi_0\in I$. 
{When no confusion can arise, we use $\Psi$ also as the leaf parameter
and write $\Sigma_{\Psi}$, $g_{\Psi}$, $d\Sigma_{\Psi}$, and
$L_{\Psi}$.}

\subsection{Statement of the main results}\label{subsec:MR}

Grad's rigidity conjecture is partly motivated by the expected
harmonicity of steady Euler flows on Bernoulli surfaces
\cite[Sec.~8]{Grad67gc}. Our first result rigorously establishes this
property under the assumptions \eqref{eq:A1} and
\eqref{eq:A2}. Throughout the remainder of the paper,
$\Omega_\Sigma\subset\mathbb{R}^3$ denotes a smooth foliated
toroidal domain.

\begin{thm}[Weighted harmonicity]\label{thm:1}
Let $\lr{u,p}$ be a nowhere-vanishing $C^1$ solution of \eqref{SE} in $\Omega_{\Sigma}$. Assume that the conditions \eqref{eq:A1} and \eqref{eq:A2} hold. Then the tangential one-form $v=i^{\ast}u$ is a weighted harmonic one-form on each $\Sigma_{\Psi}$. Namely, 
\eq{
dv=0,\qquad d\ \tilde{\ast}\ v=0\qquad \text{\rm on}~~ \Sigma_{\Psi}.  \label{eq:wharmonic0}
}
\end{thm}

Our second result constructs two globally defined solenoidal vector
fields on $\Omega_\Sigma$ whose restrictions to each leaf form a
basis of the space of weighted harmonic one-forms.

\begin{thm}[Existence of a lifted weighted harmonic basis]\label{t:basis}
Let $\tilde{\mathcal{H}}^{1}(\Sigma_{\Psi})$ denote the space of weighted harmonic one-forms on $\Sigma_{\Psi}$. There exist two vector fields $\bol{\xi}^{i}$ for $i=1,2$ such that 
\begin{align}
d\iota_{\bol{\xi}^{i}}dV=0,\quad \iota_{\bol{\xi}^{i}}d\Psi=0,\quad \textrm{span} \{i^{*}\bol{\xi}^{1\flat},i^{*}\bol{\xi}^{2\flat}\}=\tilde{\mathcal{H}}^{1}(\Sigma_{\Psi}), \label{eq:xicondition}
\end{align}
where $\bol{\xi}^{i\flat}=g_{\textrm{Euc}}(\bol{\xi}^{i},\cdot)$ is the metric-dual to $\bol{\xi}^{i}$.
\end{thm}

We can now state the main result of this paper. We regard both the
flux function $\Psi$, which determines the toroidal foliation, and
the steady Euler flow $(\bol{u},p)$ as unknown. Under the assumptions
\eqref{eq:A1} and \eqref{eq:A2}, we obtain an
explicit characterization of the foliation structure without
imposing any Euclidean symmetry. Specifically, the velocity field is
represented globally in terms of a lifted basis of weighted harmonic
one-forms, while the normal component of the Euler equations reduces
to a single scalar equation for the flux function.

\begin{thm}
[Foliation structure]\label{thm:lnGS}
Let $\Omega_{\Sigma}\subset \mathbb{R}^{3}$ be a smooth foliated toroidal domain with the flux function $\Psi$. Let $(u,p)$ be a nowhere-vanishing
$C^1$ solution of the steady Euler equations \eqref{SE} in $\Omega_{\Sigma}$. Let $\bol{\xi}^{i}$ be the vector fields satisfying \eqref{eq:xicondition}. Assume that the conditions \eqref{eq:A1} and \eqref{eq:A2} hold. Then the following hold in $\Omega_{\Sigma}$:
\begin{itemize}
\item[(i)] (Tangential flow representation). There exist functions $c_i=c_i(\Psi)$ for $i=1,2$ such that $\bol{u}$ is globally expressed as 
\eq{
\bol{u}
=
c_1\lr{\Psi}\bol{\xi}^1+c_2\lr{\Psi}\bol{\xi}^2.
\label{harmu1}
}
For some particular choice of $\bol{\xi}^{i}$, the representation can be written explicitly as  
\begin{align}
\bol{u}=h^{ij}U_j\partial_i,\quad U_j
=c_i(\Psi)
\left(
P_j^i
+
\p_j
L^{-1}
\p_k(Q^{kl}P_l^i)
\right).  \label{UA}
\end{align}
\item[(ii)] (Normal flux equation).
The flux function $\Psi$ satisfies the single scalar equation 
\eq{
{
h^{ij}U_j
\left(
\p_{\Psi}U_i
-
\p_i
\left(
g_{\Psi \theta^k}h^{kl}U_l
\right)
\right)
=p'(\Psi).
}
\label{lnGS}
}
\end{itemize}
Conversely, for prescribed functions $c_1(\Psi)$, $c_2(\Psi)$, and
$p(\Psi)$, any flux function $\Psi$ satisfying \eqref{lnGS} yields a
solution of \eqref{SE} satisfying \eqref{eq:A1}-\eqref{eq:A2}, with
$\bol{u}$ defined by \eqref{UA}.
\end{thm}

\begin{remark}
If the angle coordinates $(\theta^{1},\theta^{2})$ in Definition \ref{n:lg} are isothermal, the normal flux equation \eqref{lnGS} takes the simpler form 
\begin{align}
-\mf{L}_{\bol{u}} u_{\Psi}
+
\displaystyle\frac{\abs{\bol{u}}}{\lambda}
\frac{\p\lr{\lambda\abs{\bol{u}}}}{\p\Psi}
=
p'(\Psi),  \label{eq:lnGSiso0}
\end{align}
where $\mf{L}_{\bol{u}}$ denotes the Lie derivative along  $\bol{u}$, $u_{\Psi}$ denotes the corariant $\Psi$-component of $\bol{u}$, and $\lambda=\lambda(\Psi,\theta^{1},\theta^{2})$ is the conformal factor of the isothermal coordinates $(\theta^{1},\theta^{2})$. See Lemma \ref{l:isonlGS}.
\end{remark}

Theorem~\ref{thm:lnGS} provides the first decomposition of the steady Euler equations \eqref{SE} into a tangential flow representation and a single scalar equation for the flux function on arbitrary invariant-torus foliations under the assumptions \eqref{eq:A1} and \eqref{eq:A2}. It complements Arnold's classical structure theorem (Theorem~\ref{thm:Arnold}) by providing an explicit characterization of admissible invariant-torus foliations, thereby answering Question~\ref{q0} within this class. Indeed, the tangential flow representation \eqref{harmu1} and the normal flux equation \eqref{lnGS} reduce to the Clebsch representation \eqref{Clebsch} and the Grad--Shafranov equation \eqref{GSeq0}, respectively, in the axisymmetric setting; see Table~\ref{tab:GS-comparison}.

\begin{thm}[Axisymmetric case]\label{t:axisymmetric}
Assume that $\Omega_{\Sigma}$ and $(u,p)$ are axisymmetirc with the cylindrical
coordinates $(r,\phi,z)$ in Theorem \ref{thm:lnGS}. Then, the vector fields
\begin{align}
\bol{\xi}^1=\nabla\Psi\times\nabla\phi,
\qquad
\bol{\xi}^2=\nabla\phi,  \label{eq:xibasis}
\end{align}
satisfy \eqref{eq:xicondition}. If $c_1(\Psi)=1$ and $c_2(\Psi)=\alpha(\Psi)$, the Clebsch representation \eqref{Clebsch} and the Grad--Shafranov equation \eqref{GSeq0} hold in $\Omega_{\Sigma}$.
\end{thm}

\begin{table}[h]
\centering
\fontsize{10}{12}\selectfont
\setlength{\tabcolsep}{2pt}
\setlength{\extrarowheight}{1pt}
\renewcommand{\arraystretch}{1.75}

\begin{tabular}{
@{}
>{\raggedright\arraybackslash\bfseries}p{0.2\textwidth}
>{\centering\arraybackslash}p{0.34\textwidth}
>{\centering\arraybackslash}p{0.30\textwidth}
@{}
}
\toprule
&
\textbf{Foliation structure\hspace{15pt}
(Theorem \ref{thm:lnGS})}
&
\textbf{Axisymmetric case (Theorem \ref{t:axisymmetric})}
\\
\midrule

Tangential flow 
representation
&
\(
\begin{aligned}[t]
\bol{u}
=
c_1\lr{\Psi}\bol{\xi}^1
+
c_2\lr{\Psi}\bol{\xi}^2
\end{aligned}
\)
&
\(
\begin{aligned}[t]
\bol{u}
&=
\nabla\Psi\times\nabla\phi
+
\alpha\lr{\Psi}\nabla\phi \\
&(\textrm{Clebsch representation})
\end{aligned}
\)
\\
\addlinespace[3pt]

Normal flux equation
&
\(
\begin{aligned}[t]
&-\mf{L}_{\bol{u}} u_{\Psi}
+
\displaystyle\frac{\abs{\bol{u}}}{\lambda}
\frac{\p\lr{\lambda\abs{\bol{u}}}}{\p\Psi}
=
p'(\Psi) \\
&\hspace{7pt} (\textrm{with isothermal coordinates})
\end{aligned}
\)
&
\(
\begin{aligned}[t]
&\Delta^\ast\Psi+\alpha(\Psi)\alpha'(\Psi)=r^2p'(\Psi) \\
&\hspace{3pt} (\textrm{Grad--Shafranov equation})
\end{aligned}
\)
\\
\bottomrule
\end{tabular}
\caption{Foliation structure of general steady Euler flows and axisymmetric case in Theorems \ref{thm:lnGS} and \ref{t:axisymmetric}.}
\label{tab:GS-comparison}
\end{table}

\begin{remark}
The tangential flow representation
\begin{align}
\bol{u}=c_1(\Psi)\nabla \Psi\times \nabla \phi+c_2(\Psi)\nabla \phi \label{Clebsch2}
\end{align}
is more intrinsic than the normalized Clebsch representation \eqref{Clebsch} {with $\alpha=\alpha\lr{\Psi}$}. If $c_1(\Psi)\neq 0$, we can normalize the coefficient $c_1(\Psi)$ by choosing a function $\Phi(\Psi)$ such that $\Phi'(\Psi)=c_1(\Psi)$. Using its inverse function $\Psi=\Psi(\Phi)$, we rewrite \eqref{Clebsch2} in the normalized form
\begin{align*}
\bol{u}=\nabla \Phi\times \nabla \phi+c_2(\Psi(\Phi))\nabla \phi,
\end{align*}
with the new flux function $\Phi$.
\end{remark}

\begin{remark}[Foliations of Beltrami flows]
It is worth noting that steady Euler flows with various foliation structures have been constructed in the class of Beltrami flows \cite{EPS12,EPS15}, namely, solutions to \eqref{SE} with constant pressure; see also Question~\ref{q0}. Such flows are obtained by solving
\begin{align}
\nabla\times\bol{u}=f\bol{u}.
\label{eq:Beltrami}
\end{align}
Unlike the steady Euler equations \eqref{SE}, the Beltrami equation \eqref{eq:Beltrami} can be formulated as an evolution equation in the direction normal to the flux surfaces \cite{EPS12,EPS15}.

The analogy becomes even closer for steady Euler flows with
nonconstant pressure when the proportionality factor $f$ is
nonconstant. In this case, the level sets $\{f=f_0\}$ play a role analogous to the Bernoulli surfaces $\{p=p_0\}$; see, for example, \cite[II]{ArnoldKh}. Enciso and Peralta-Salas \cite{EPS16} proved that the restriction of a Beltrami field to each regular level set $\{f=f_0\}$ is a weighted harmonic one-form, while its dependence in the normal direction is still governed by an evolution equation; see also \cite{Abe2022,Abe22}.

By contrast, Theorem~\ref{thm:lnGS} shows that, for the class of
steady Euler flows considered here, the normal dependence is
completely characterized by the normal flux equation
\eqref{lnGS}, rather than by an evolution equation.
\end{remark}

\if{
\begin{itemize}
\item (iii) solutions arising in modified or anisotropic-pressure models, where the
right-hand side of the force-balance equation in \eqref{SE} is replaced by a
more general pressure force. With regard to (iii),
the work \cite{Sato23} shows that introducing an anisotropy in the form of an
integrating factor, so that the force-balance equation takes the form
$\iota_{\bol{u}}du=-\lambda(\bol{x})dp$, 
is sufficient to produce smooth equilibria in hollow toroidal volumes of
general geometry.

Outside the preceding classes, available existence results typically rely either
on weaker notions of regularity or on modified forms of the equilibrium
equations. Relevant recent examples include Sobolev-regular equilibria,
typically with \(\bol{u}\in H^1\), obtained through relaxation schemes, such as
the Voigt--MHD system of \cite{Pasqualotto23} and the resistive
magnetic-relaxation equations with random forcing studied in \cite{Abe26};
H\"older-continuous steady Euler flows with prescribed topology, constructed
through the topology-preserving convex integration scheme of
\cite{EncisoPenafielTomasPeraltaSalas2025TopologyPreserving};
stepped-pressure equilibria \cite{Bruno96,Dewar15,Enciso25}; smooth equilibria
obtained by adding an external force to the force-balance equation
\cite{Constantin21}; smooth solutions obtained by relaxing the incompressibility
constraint on \(\bol{u}\) \cite{Sato25}; solutions constructed after allowing
more general ambient metric tensors \cite{Cardona25}; and iterative schemes
based on multivalued Clebsch potentials \cite{Sato24}. For comparison, we also
note the recent result \cite{PSPP26} showing that, even for rotationally
symmetric solid tori, minimizers of Woltjer's variational principle,
equivalently first Amp\`erian curl eigenfields solving
\(\nabla\times \bol{B}=\lambda \bol{B}\) with the appropriate boundary
conditions, need not inherit the rotational symmetry of the domain.
\end{itemize}
}\fi

\subsection{Ideas of the proof}

We briefly outline the proofs of
Theorems~\ref{thm:1}--\ref{thm:lnGS}, and \ref{t:axisymmetric}.
We first establish the weighted harmonicity
\eqref{eq:wharmonic0} of the velocity field on each invariant
torus. For this purpose, we construct local isothermal coordinates
$(\theta^1,\theta^2)=(\chi,\phi)$ on each leaf $\Sigma_\Psi$ and
the corresponding foliation-adapted coordinates
$(\Psi,\chi,\phi)$ in $\Omega_\Sigma$. These coordinates simplify
the metric computations and allow us to derive a local component
form of the Euler equations. We then interpret each leaf as a
Bakry--\'Emery manifold and deduce the weighted harmonicity from the
tangential components of the Euler equations.

The main step is the construction of the lifted basis in
Theorem~\ref{t:basis}. We seek smoothly varying weighted
harmonic one-forms $\tau^i=\tau^i(\Psi)$, $i=1,2$, in the form
\[
\tau^i=(P^i_j+\partial_jF^i)d\theta^j,
\]
where $P^i=P^i_jd\theta^j$ are linearly independent period covectors.
The forms $\tau^i$ are automatically closed. Their weighted
co-closedness reduces to a leafwise elliptic equation
\[
L F^i=\partial_k(Q^{kl}P^i_l).
\]
Since $L$ is invertible on average-zero functions, this equation
determines a smoothly varying basis of weighted harmonic one-forms.
Taking the metric duals and lifting them through the foliation yields
the desired solenoidal vector fields $\bol{\xi}^1$ and
$\bol{\xi}^2$.

Once this lifted basis has been constructed, the tangential flow
representation follows by expanding the tangential velocity one-form
in that basis. Substituting this representation into the normal
component of the Euler equation,
\[
\iota_{\bol{u}}\iota_Tdu=p'(\Psi),
\]
then yields the normal flux equation \eqref{lnGS}.\\

We conclude this introduction by mentioning a related work of
K.~de Lacy \cite{deLacy2026}, who also exploits a weighted
harmonic property closely related to Theorem~\ref{thm:1} in the
study of MHS equilibrium approximation.

\subsection{Organization of this paper}

The remainder of this paper is organized as follows.
In Section~\ref{sec:red}, we construct foliation-adapted isothermal
coordinates and derive a component representation of the equations
\eqref{SE}. In Section~\ref{sec:harm}, we introduce weighted harmonic
one-forms on closed Bakry--\'Emery manifolds and prove that the
restriction of the velocity one-form to each invariant torus is
weighted harmonic (Theorem~\ref{thm:1}). In Section \ref{s4}, we construct the lifted weighted harmonic vector fields (Theorem~\ref{t:basis}). In
Section~\ref{sec:global}, we derive the tangential flow representation
\eqref{harmu1} and the normal flux equation \eqref{lnGS}, thereby
proving Theorem~\ref{thm:lnGS}. We also show the representation \eqref{eq:lnGSiso0} in the isothermal coordinates and the reduction in the axisymmetric case (Theorem \ref{t:axisymmetric}).

\subsection{Acknowledgments} 
N.S. would like to acknowledge helpful discussions with K. de Lacy during his visit to NIFS, as well as with D. Pfefferlé and D. Perrella. 
The research of N.S. was partially supported by JSPS KAKENHI Grant 
No. 25K07267, No. 22H04936, and No. 24K00615. 
The research of K.A. was partially supported by the JSPS through the Grant in Aid for Scientific Research (C) 24K06800, MEXT Promotion of Distinctive Joint Research Center Program JPMXP0723833165, and Osaka Metropolitan University Strategic Research Promotion Project (Development of International Research Hubs).

\section{The foliation-adapted isothermal
coordinates}\label{sec:red}

We construct a local coordinate system adapted to the toroidal
foliation of the domain $\Omega_{\Sigma}$. The construction is based on local
isothermal coordinates on the leaves $\Sigma_{\Psi}$. We refer to, e.g., \cite{doCarmo,Lee2012} for background on embedded submanifolds, induced metrics, and regular level
sets.

\subsection{Coordinate construction}\label{subsec:iso}

Let \(\Omega_{\Sigma}\subset\mathbb{R}^3\) be a foliated toroidal
domain in Definition \ref{def:marked-foliation} with a smooth submersion $\Psi$. We take a coordinate system $(\Psi,\chi,\phi)$ defined on local regions in $\Omega_{\Sigma}$ such that $(\chi,\phi)$ is isothermal on each local leaf slice; see Figure ~\ref{fig2}.

\begin{figure}[h]
\centering
\includegraphics[width=0.5\textwidth]{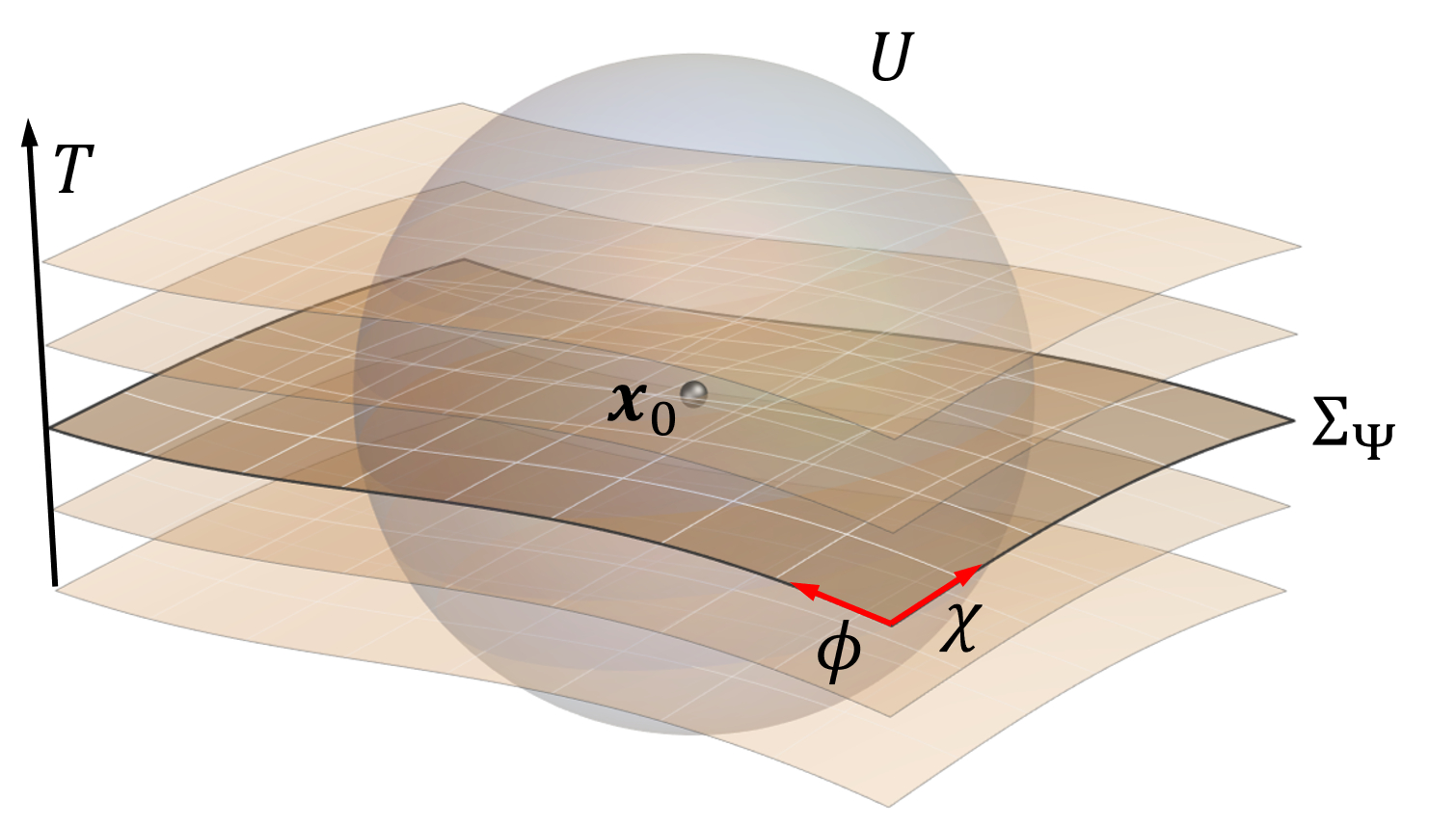}
\caption{
Schematic of a foliation-adapted neighborhood. The local neighborhood
\(U\) intersects the level surfaces \(\Sigma_{\Psi}\). Around
a point \(\bol{x}_0\in\Sigma_{\Psi}\), the coordinates \((\Psi,\chi,\phi)\) are
chosen so that \((\chi,\phi)\) are local isothermal coordinates on each leaf.
}
\label{fig2}
\end{figure}


\begin{prop}[Leafwise isothermal coordinates]\label{p:iso1}
For an arbitrary $\bol{x}_0\in \Omega_{\Sigma}$, there exist an open set $U\subset \Omega_{\Sigma}$ and the coordinates $(\Psi,\chi,\phi)$ on $U$ satisfying 
\begin{align}
&g_\Psi\big|_{U_\Psi}
=
\lambda\lr{\Psi,\chi,\phi}^2
\lr{d\chi^2+d\phi^2},  \label{eq:iso10}\\ 
&\iota_{\p_{\Psi}}d\Psi=1,
\qquad
\iota_{\p_{\Psi}}d\chi=0,
\qquad
\iota_{\p_{\Psi}}d\phi=0,  \label{eq:iso20}
\end{align}
with some positive function $\lambda(\Psi,\chi,\phi)>0$.
\end{prop}

\begin{proof}
We set \(\Psi_0=\Psi(\bol{x}_0)\) and choose local ambient coordinates
\(\lr{x^1,x^2,x^3}\) near the point $\bol{x}_0\in \Omega_{\Sigma}$. Since \(d\Psi_{\bol{x}_0}\neq0\), by relabeling the coordinates if necessary, we may assume that 
\begin{align*}
\frac{\partial\Psi}{\partial x^3}(\bol{x}_0)\neq0.
\end{align*}
The map $U\ni \bol{x}\longmapsto \bol{y}=\lr{\Psi(\bol{x}),x^1,x^2}\in I\times \mathbb{R}^{2}
$ 
has nonzero Jacobian at \(\bol{x}_0\). By the inverse function theorem, there exist coordinates $
\lr{\Psi,y^1,y^2}$ 
on \(U\) such that $(y^1,y^{2})=(x^1,x^2)$. The induced metric on each local leaf slice $U_\Psi=U\cap\Sigma_\Psi$ has
the form
\eq{
g_\Psi\big|_{U_\Psi}
=
h_{ij}\lr{\Psi,y^{1},y^{2}}\,dy^idy^j,
\qquad i,j=1,2,
}
with smooth coefficients \(h_{ij}(\Psi,y^{1},y^{2})\).


We apply the local isothermal-coordinate theorem for oriented surfaces \cite[\S3.11]{Jost} and deduce the existence of the functions 
\begin{align}
\lr{\chi,\phi}
=
\lr{\chi\lr{\Psi,y^{1},y^{2}},\phi\lr{\Psi,y^{1},y^{2}}}
\end{align}
such that \eqref{eq:iso10} holds for some positive function  $\lambda(\Psi,\chi,\phi)>0$. Since the leafwise Jacobian
$\partial\lr{\chi,\phi}/{\partial (y^{1},y^{2})}$  
is nonzero after shrinking \(U\), the map
$\lr{\Psi,y^1,y^2}\longmapsto\lr{\Psi,\chi,\phi}$ 
is a local diffeomorphism. Thus \(\lr{\Psi,\chi,\phi}\) is a local
coordinate on $U$ isothermal on each local leaf slice $U_{\Psi}$ and satisfies \eqref{eq:iso20}. 
\end{proof}

\if{
Hence, locally, \(\partial_\Psi\) is a transverse lift of the leaf-label
direction. 
Globally, however, we do not assume that these local coordinate vector fields
patch together to define a distinguished global \(\partial_\Psi\). 
Instead, global statements are formulated using an arbitrary smooth
transverse vector field \(T\) satisfying
$\iota_Td\Psi=1$.
All formulas involving \(\lr{\chi,\phi}\), \(\lambda\), \(J\), and
\(\partial_\Psi\) are therefore local coordinate formulas on such a
neighbourhood \(U\subset\Omega_\Sigma\). 
The global statements are independent of this choice and are expressed
intrinsically in terms of \(g_\Psi\), the leafwise measure 
$d\Sigma=\left.\iota_TdV\right|_{\Sigma_\Psi}$, and the weighted harmonic space 
\(\tilde{\ml{H}}^1\lr{\Sigma_\Psi}\).
}\fi

\begin{prop}[Coordinate expression of the metric]
Let $\lr{\Psi,\chi,\phi}$ denote the local coordinates in Proposition \ref{p:iso1}. Then, the metric and the volume form are expressed as 
\begin{align}
&g_{\mathrm{Euc}}
=
g_{\Psi\Psi}\,d\Psi^2
+
2g_{\Psi\chi}\,d\Psi\,d\chi
+
2g_{\Psi\phi}\,d\Psi\,d\phi
+
\lambda^{2}\,(d\chi^{2}+d\phi^{2}),  \label{eq:metriciso} \\
&dV=J\, d\Psi\w d\chi\w  d\phi,\qquad
J=\lambda\sqrt{\lambda^2g_{\Psi\Psi}-g_{\Psi\phi}^2-g_{\Psi\chi}^2}.\label{eq:volueiso}
\end{align}
The inverse metric is given by 
\begin{align}
g_{\mathrm{Euc}}^{-1}
={}
\frac{\lambda^{4}}{J^{2}}\,\partial_{\Psi}^{2}
-\frac{2\lambda^{2}g_{\Psi\chi}}{J^{2}}\,
\partial_{\Psi}\partial_{\chi}
-\frac{2\lambda^{2}g_{\Psi\phi}}{J^{2}}\,
\partial_{\Psi}\partial_{\phi} 
+
\left(
\frac{1}{\lambda^{2}}
+\frac{g_{\Psi\chi}^{2}}{J^{2}}
\right)\partial_{\chi}^{2}
+\frac{2g_{\Psi\chi}g_{\Psi\phi}}{J^{2}}\,
\partial_{\chi}\partial_{\phi} 
+
\left(
\frac{1}{\lambda^{2}}
+\frac{g_{\Psi\phi}^{2}}{J^{2}}
\right)\partial_{\phi}^{2}.  \label{eq:inverseiso}
\end{align}
\end{prop}

\begin{proof}
The expression \eqref{eq:metriciso} follows from \eqref{eq:metric}. In the matrix form,  
\[
g_{\mathrm{Euc}}
=
\begin{pmatrix}
g_{\Psi\Psi} & g_{\Psi\chi} & g_{\Psi\phi} \\
g_{\Psi\chi} & \lambda^2 & 0 \\
g_{\Psi\phi} & 0 & \lambda^2
\end{pmatrix},
\]
and \eqref{eq:volueiso} follows from $J=\sqrt{\textrm{det}\ g_{\mathrm{Euc}}}$. We use the block matrix  
\[
g_{\mathrm{Euc}}=
\begin{pmatrix}
A & B\\
C & D
\end{pmatrix},
\quad A=g_{\Psi\Psi},\quad B=(g_{\Psi\chi}, g_{\Psi\phi}),\quad C={}^{t}B,\quad D=\lambda^{2}I,
\]
and set $S=A-BD^{-1}C=J^{2}\lambda^{-4}$, where $I$ denotes the identity matrix. Using the Schur complement formula,
\[
\begin{pmatrix}
A & B\\
C & D
\end{pmatrix}
=
\begin{pmatrix}
I & BD^{-1}\\
0 & I
\end{pmatrix}
\begin{pmatrix}
S & 0\\
0 & D
\end{pmatrix}
\begin{pmatrix}
I & 0\\
D^{-1}C & I
\end{pmatrix},
\]
we compute 
\begin{align*}
g_{\mathrm{Euc}}^{-1}
=
\begin{pmatrix}
I & 0\\
-D^{-1}C & I
\end{pmatrix}
\begin{pmatrix}
S^{-1} & 0\\
0 & D^{-1}
\end{pmatrix}
\begin{pmatrix}
I & -BD^{-1}\\
0 & I
\end{pmatrix}
=
\begin{pmatrix}
\lambda^{4}J^{-2} & -\lambda^{2}J^{-2}B\\
-(\lambda^{2}J^{-2}){}^{t}B & (J^{-2}) {}^{t}BB+\lambda^{-2}I
\end{pmatrix}.
\end{align*}
We obtain the expression \eqref{eq:inverseiso}. 
\end{proof}

\subsection{The local representation of the Euler equations}\label{sec:6}

In the local coordinates $(\Psi,\chi,\phi)$ on $U$, we express the velocity field $\bol{u}$ and its dual one-form $u=g_{\textrm{Euc}}(\bol{u},\cdot)$ as 
\begin{align}
\bol{u}&=u^{\Psi}(\Psi,\chi,\phi)\partial_\Psi+u^{\chi}(\Psi,\chi,\phi)\partial_\chi+u^{\phi}(\Psi,\chi,\phi)\partial _\phi,   \label{eq:tangentialiso}\\
u&=u_\Psi(\Psi,\chi,\phi)d\Psi+u_\chi(\Psi,\chi,\phi) d\chi+u_\phi(\Psi,\chi,\phi) d\phi.
\label{eq:one-formiso}
\end{align}

\begin{prop}
Assume that \eqref{eq:A1} holds. Then, we have  
\begin{align}
\bol{u}&=u^{\chi}(\Psi,\chi,\phi)\partial_{\chi}+u^{\phi}(\Psi,\chi,\phi)\partial_\phi,   \label{eq:tangential}\\
u&=\left(g_{\Psi\chi}u^{\chi}(\Psi,\chi,\phi)+g_{\Psi\phi}u^{\phi}(\Psi,\chi,\phi)\right)d\Psi+\lambda^{2}u^{\chi}(\Psi,\chi,\phi)d\chi +\lambda^{2}u^{\phi}(\Psi,\chi,\phi)d\phi.
\label{eq:one-form}
\end{align}
\end{prop}

\begin{proof}
This follows from the metric expression \eqref{eq:metriciso}.
\end{proof}

\begin{lemma}\label{l:SEisoA}
Let $(u,p)$ be a $C^{1}$ solution to the steady Euler equation \eqref{SE} on $U$. Assume that the conditions \eqref{eq:A1} and \eqref{eq:A2} hold. Then, $(u_{\chi}, u_{\phi},p)$ satisfies the following system: 
\begin{equation}
\begin{aligned}
u^\chi
\left(
\partial_\Psi u_\chi-\partial_\chi u_\Psi
\right)
+
u^\phi
\left(
\partial_\Psi u_\phi-\partial_\phi u_\Psi
\right)
&=
\partial_\Psi p,
\\[1ex]
u^\phi
\left(
\partial_\chi u_\phi-\partial_\phi u_\chi
\right)
&=0,
\\[1ex]
u^\chi
\left(
\partial_\chi u_\phi-\partial_\phi u_\chi
\right)
&=0,
\\[1ex]
\partial_\chi(Ju^\chi)
+
\partial_\phi(Ju^\phi)
&=
0.
\end{aligned}
\label{eq:SEisoflux0}
\end{equation}
\end{lemma}

\begin{proof}
The steady Euler equations \eqref{SE} in the coordinates $(\Psi,\chi,\phi)$ are written as 
\begin{equation}
\begin{aligned}
u^\chi
\left(
\partial_\Psi u_\chi-\partial_\chi u_\Psi
\right)
+
u^\phi
\left(
\partial_\Psi u_\phi-\partial_\phi u_\Psi
\right)
&=
\partial_\Psi p,
\\[1ex]
u^\Psi
\left(
\partial_\Psi u_\chi-\partial_\chi u_\Psi
\right)
-
u^\phi
\left(
\partial_\chi u_\phi-\partial_\phi u_\chi
\right)
&=
-\partial_\chi p,
\\[1ex]
u^\Psi
\left(
\partial_\Psi u_\phi-\partial_\phi u_\Psi
\right)
+
u^\chi
\left(
\partial_\chi u_\phi-\partial_\phi u_\chi
\right)
&=
-\partial_\phi p,
\\[1ex]
\partial_\Psi(Ju^\Psi)
+
\partial_\chi(Ju^\chi)
+
\partial_\phi(Ju^\phi)
&=
0.
\end{aligned}
\label{eq:SEiso}
\end{equation}

The conditions \eqref{eq:A1} and \eqref{eq:A2} imply that $u^{\Psi}=0$ and $\partial_{\chi}p=\partial_{\phi}p=0$ and \eqref{eq:SEisoflux0} follows.
\end{proof}

\begin{prop}
Assume in addition in Lemma \ref{l:SEisoA} that $u$ is nowhere-vanishing on $U$. Then 
\begin{align}
\partial_\chi u_\phi-\partial_\phi u_\chi&=0,  \label{eq:irrotationaliso}\\
\partial_\chi\left(\frac{J}{\lambda^{2}}u_\chi\right)
+
\partial_\phi\left(\frac{J}{\lambda^{2}}u_\phi\right)
&=
0.  \label{eq:divfreeiso}
\end{align}
\end{prop}

\begin{proof}
The components
$u^{\chi}$ and $u^{\phi}$ do not vanish simultaneously by $u\neq 0$ and \eqref{eq:irrotationaliso} follows from the second and the third equations of \eqref{eq:SEisoflux}. The condition \eqref{eq:divfreeiso} follows from \eqref{eq:tangential}-\eqref{eq:one-form} and the last equation of \eqref{eq:SEisoflux0}. 
\end{proof}

\section{Weighted harmonicity
  of tangential flows}\label{sec:harm}

We recall Bakry-Émery manifolds and define weighted harmonic one-forms on the torus $\Sigma_{\Psi}$. We then show that the tangential restriction of the velocity one-form of steady Euler flows \eqref{SE} is weighted harmonic (Theorem \ref{thm:1}).

\subsection{Bakry--Émery manifolds}\label{sec:weighted}

\begin{mydef}[Bakry--Émery manifold]\label{def:BEM}

Let \(M\) be a smooth manifold with a Riemannian metric \(g\). Let \(f\in C^\infty(M)\) be a positive function. We call the triple \((M,g,f)\) a Bakry--Émery manifold. 
\end{mydef}

We define the weighted codifferential and the Hodge-star operator with the function \(f\). 

\begin{mydef}[Weighted codifferential and Hodge-star operators]\label{def:codiff}
Let \((M,g,f)\) be a closed Bakry--Émery manifold of dimension \(n\). Let \(\Lambda^k\lr{M}\) denote the space of smooth \(k\)-forms on \(M\). For each \(k\), we set the weighted codifferential
\eq{
\tilde{\delta}:\Lambda^k(M)\to\Lambda^{k-1}(M)
}
as the formal adjoint of $d:\Lambda^{k-1}(M)\to\Lambda^k(M)$ with respect to the weighted \(L^2\)-inner product
\eq{
\langle \alpha,\beta\rangle_{L^2_f}
=
\int_M \langle \alpha,\beta\rangle\, f\,dV_g,
}
where \(\langle\cdot,\cdot\rangle\) denotes the pointwise pairing induced by \(g\). We define the weighted Hodge-star operator by 
\eq{
\tilde{\ast}=f\ \ast:\Lambda^k(M)\to\Lambda^{n-k}(M).
\label{eq:wHodgestar}}
\end{mydef}

\begin{prop}
The weighted codifferential \(\tilde{\delta}\) is expressed as 
\begin{align}
\tilde{\delta}\beta
=\frac{1}{f}\,\delta(f\beta)
=\frac{1}{f}\,(-1)^{n(k+1)+1}\ast d\ \tilde{\ast}\ \beta,\quad \beta\in\Lambda^k(M),
\label{wco} 
\end{align}
where $\delta:\Lambda^k(M)\to\Lambda^{k-1}(M)$ denotes the usual codifferential.
\end{prop}

\begin{proof}
For \(\gamma\in\Lambda^{k-1}(M)\) and \(\beta\in\Lambda^k(M)\), 
\begin{align*}
\langle d\gamma,\beta\rangle_{L^2_f}
=
\int_M \langle d\gamma,\beta\rangle\, f\,dV_g
=
\int_M \langle d\gamma,f\beta\rangle\, dV_g
=
\int_M \langle \gamma,\delta(f\beta)\rangle\, dV_g
=
\int_M \langle \gamma,\frac{1}{f}\delta(f\beta)\rangle\, f\,dV_g.
\end{align*}
Thus the first identity in \eqref{wco} follows. The second follows from $\delta\beta= (-1)^{n(k+1)+1}\ast d\ast\beta$.
\end{proof}

We use a weighted version of the classical Hodge theorem, e.g., \cite[Eq.~(2.13)]{Lott03}, \cite[Sec.~2.2]{Branding24}. Since the differential \(d\) is unchanged, the de Rham cohomology groups are the same as in the unweighted case; the role of the weight is only to modify the formal adjoint of \(d\), and hence the notion of harmonic representative.

\begin{prop}[Weighted harmonic forms]\label{prop:BEM}
Let \((M,g,f)\) be a closed Bakry--Émery manifold. Then, 
\eq{
H^k_{\rm dR}\lr{M}
=
\frac{{\rm ker}\lr{d:\Lambda^k\lr{M}\rightarrow\Lambda^{k+1}\lr{M}}}
{{\rm im}\lr{d:\Lambda^{k-1}\lr{M}\rightarrow\Lambda^{k}\lr{M}}}\cong \tilde{\ml{H}}^k(M)
=
\lrc{\xi\in\Lambda^k(M):d\xi=\tilde{\delta}\xi=0}.
}
In particular, every de Rham cohomology class admits a unique weighted harmonic representative.
\end{prop}

\subsection{Weighted harmonicity}

Let $\Omega_{\Sigma}\subset \mathbb{R}^{3}$ be a smooth foliated toroidal
domain with a smooth flux function $\Psi$ as in Definition \ref{def:marked-foliation}. Let \(T\) be a smooth transverse vector field satisfying \eqref{eq:tran}. We define the coordinates, metric, and measure as in Definition \ref{n:lg}.

\begin{prop}
The measure $d\Sigma_{\Psi}$ in  \eqref{eq:MM} is independent of the choice of $T$ satisfying \eqref{eq:tran}.
\end{prop}

\begin{proof}
Let \(T'\) be another smooth transverse vector field satisfying \eqref{eq:tran}. Then, \(T'=T+W\) for some vector field \(W\) tangential to each \(\Sigma_\Psi\). Then, for the inclusion map $i:\Sigma_{\Psi}\hookrightarrow\mathbb{R}^3$,
\begin{align*}
\left.\iota_{T'}dV\right|_{\Sigma_\Psi}
-
\left.\iota_TdV\right|_{\Sigma_\Psi}
=\left.\iota_WdV\right|_{\Sigma_\Psi}=i^\ast\lr{\iota_WdV}=0.
\end{align*}
Hence the leafwise
measure $
d\Sigma$ 
is independent of the tangential ambiguity in the choice of \(T\).
\end{proof}

We set  
\begin{align}
f_{\Psi}=\frac{d\Sigma_{\Psi}}{dV_{g_{\Psi}}},
\end{align}
by the volume form $dV_{g_{\Psi}}$ on $\Sigma_{\Psi}$ induced from the metric $g_{\Psi}=i^{*}g_{\textrm{Euc}}$ and consider the Bakry--Émery manifold $(\Sigma_{\Psi},g_{\Psi},f_{\Psi})$.

\begin{proof}[Proof of Theorem \ref{thm:1}]
It suffices to show \eqref{eq:wharmonic0} in the local  isothermal coordinates $\lr{\Psi,\chi,\phi}$ on $U$. It follows from \eqref{eq:iso1} that  
\begin{align*}
g_{\Psi}&=\lambda^2(d\chi^2+d\phi^2),\\
dV_{g_{\Psi}}&=\lambda^{2}d\chi\wedge d\phi.
\end{align*}
We take $T=\partial_{\Psi}$ and observe from the metric expression \eqref{eq:volueiso} that  
\begin{align*} d\Sigma_{\Psi}=\iota_{T}dV\big|_{\Sigma_{\Psi}}=J\,d\chi\w d\phi.
\end{align*}
Thus ${f_{\Psi}}=J^{2}\lambda^{-2}$. By the definition of weighted Hodge-star operator \eqref{eq:wHodgestar},
\begin{align*}
\tilde *\,d\chi=\frac{J^{2}}{\lambda^{2}}\,d\phi,\qquad 
\tilde *\,d\phi=-\frac{J^{2}}{\lambda^{2}}\,d\chi.
\end{align*}
For the tangential one-form $v=u_{\chi}d\chi+u_{\phi}d\phi$, it follows from \eqref{eq:irrotationaliso} that  
\begin{align*}
dv=\lr{\frac{\p u_{\phi}}{\p\chi}-\frac{\p u_{\chi}}{\p\phi}}\,d\chi\wedge d\phi=0.
\end{align*}
For $\tilde *\,v
=J^{2}\lambda^{-2} u_{\chi}\,d\phi - J^{2}\lambda^{-2}\,d\chi$, it follows from \eqref{eq:divfreeiso} that 
\begin{align*}
d\ \tilde{*}\,v
=\lr{\frac{\p}{\p\chi}\lr{\frac{J}{\lambda^2}u_{\chi}}+\frac{\p}{\p\phi}\lr{\frac{J}{\lambda^2}u_{\phi}}}\,d\chi\wedge d\phi
=0.
\end{align*}
Thus $v$ is $\tilde{\ast}$-harmonic. 
\end{proof}

\begin{remark}\label{r:dim}
On each leaf $\Sigma_{\Psi}$, the space of weighted harmonic one-forms $\tilde{\mathcal{H}}^{1}(\Sigma_{\Psi})$ is two-dimensional. Indeed, $\tilde{\mathcal{H}}^{1}(\Sigma_{\Psi})$ is isomorphic to $H^{1}_{\textrm{dR}} (\Sigma_{\Psi})$ by Proposition \ref{prop:BEM}. Since $\Sigma_{\Psi}\cong \mathbb{T}^{2}$, $H^{1}_{\textrm{dR}} (\Sigma_{\Psi})\cong H^1_{\rm dR}(\mathbb{T}^2)\cong \mathbb{R}^2$ by the K\"unneth formula, e.g., \cite{BottTu1982}.
\end{remark}

\section{Lifted weighted harmonic vector fields}\label{s4}

We show Theorem \ref{t:basis} by constructing weighted harmonic one-forms on tori $\Sigma_{\Psi}$ and lifting their metric-dual tangential vector fields on $\Sigma_{\Psi}$ into solenoidal vector fields in $\Omega_{\Sigma}$.

\subsection{Weighted harmonic one-forms}

We construct a basis of $\tilde{\mathcal{H}}^{1}(\Sigma_{\Psi})$ using the inverse of a uniformly elliptic operator on the periodic torus $D=(\mathbb{R}/\mathbb{Z})^{2}$. It suffices to construct two linearly independent weighted harmonic one-forms on $\Sigma_{\Psi}$ by Remark \ref{r:dim}.

\begin{prop}\label{p:elliptic}
Let $D=(\mathbb{R}/\mathbb{Z})^{2}$ be a periodic torus. Let $Q(\vartheta)$ be a 
{smooth uniformly 
positive definite}  symmetric tensor on $D$. Let $L^{2}_{\textrm{av}}(D)$ denote the space of all average-zero functions in $L^{2}(D)$. Set the operator $L$ on $D$ by
\begin{align}
-LF=\textrm{div}_{\vartheta}(Q\nabla_{\vartheta}F).   \label{eq:operator}
\end{align}
For $G\in L^{2}_{\textrm{av}}(D)$, there exists a unique solution $F\in H^{2}{\lr{D}}\cap L^{2}_{\textrm{av}}(D)$ to
\begin{align}
{LF=G.}
\end{align}
Namely, the inverse operator $L^{-1}: L^{2}_{\textrm{av}}(D)\to H^{2}{\lr{D}}\cap L^{2}_{\textrm{av}}(D)$ exists. If $G$ is smooth on ${D}$, so is $F$. 
\end{prop}

\begin{proof}
{The existence in $H^1_{\rm av}(D)
=
H^1(D)\cap L^2_{\rm av}(D)$ follows from the Lax-Milgram theorem.}  The regularity assertion follows from repeated applications of elliptic regularity estimates, e.g., \cite{GilbargTrudinger2001}.
\end{proof}



\begin{lemma}\label{lem:LBWH}
There exist a family of one-forms $\tau^{i}=\tau^{i}(\Psi)$ for $i=1,2$ on $\Sigma_{\Psi}$ such that $(\tau^{1},\tau^{2})$ is a basis on $\tilde{\mathcal{H}}^{1}(\Sigma_{\Psi})$ for each $\Sigma_{\Psi}$ and is smoothly varying in $\Psi$.
\end{lemma}

\begin{proof}
We take a basis \(\{\gamma_1,\gamma_2\}\) of \(H_1(\Sigma_\Psi;\mathbb Z)\) and chose $(\theta^{1},\theta^{2})$ so that
\eq{
        \oint_{\gamma_j} d\theta^i=\delta^i_j.
}
We use the adapted coordinates \((\Psi,\theta^1,\theta^2)\) in Definition \ref{def:adapted}. We seek the desired family of one-forms $\tau^{i}$ for $i=1,2$ in the closed form
\eq{
        \tau^i
        =\tau^{i}_jd\theta^{j}=
        \left(P_j^i+\partial_j F^i \right)d\theta^j, \label{eq:tau}
}
so that $d\tau^{i}=0$ with linearly independent period covectors $P^{i}=P^{i}_{j}d\theta^{j}$ in Definition \ref{n:lg} and some function $F=F({\Psi},\theta^{1},\theta^{2})$. The weighted co-closedness condition \(d\ \tilde{*}\ \tau=0\) is equivalent to the condition 
\begin{align}
\partial_i\left(Q^{ij}\tau_j\right)
        =\partial_i\left(
        Q^{ij}
        \left(P_j^i+\partial_j F^i\right)
        \right)
        =0,  \label{eq:divergence}
\end{align}
{with $Q^{ij}$ the smooth uniformly positive definite} symmetric tensor {defined in eq.} \eqref{eq:Ell}. Using the uniformly elliptic operator ${L=L_{\Psi}}$
we express \eqref{eq:divergence} as 
\eq{
L F^i
=\partial_k\left(Q^{kj}P_j^i\right).
}
By setting $F^i= L^{-1}\partial_k\left(Q^{kj}P_j^i\right)$, we obtain two weighted harmonic one-forms on $\Sigma_{\Psi}$ which are smooth in $\Psi$.

It remains to show that $(\tau^{1},\tau^{2})$ is a basis on $\tilde{\mathcal{H}}^{1}(\Sigma_{\Psi})$. Since $\tilde{\mathcal{H}}^{1}(\Sigma_{\Psi})$ is two-dimensional as noted in Remark \ref{r:dim}, we show that 
{$\tau^1$ and $\tau^2$ are} 
linearly independent. We set 
\begin{align}
a_1\tau^{1}+a_2\tau^{2}=0,  \label{eq:lindep}
\end{align}
for some constants $a_1$ and $a_2$. Since the integral of an exact form on a closed curve vanishes, integrating \eqref{eq:tau} yields 
\begin{align*}
\oint_{\gamma_j} \tau^{i} =P^{i}_j.
\end{align*}
By integrating \eqref{eq:lindep}, we find that 
\begin{align*}
\begin{pmatrix}
P^{1}_{1} & P^{2}_{1}\\
P^{1}_{2} & P^{2}_{2}
\end{pmatrix}
\begin{pmatrix}
a_1 \\
a_{2} 
\end{pmatrix}
=0.
\end{align*}
This matrix is invertible since $P^{1}$ and $P^{2}$ are linearly independent. Thus $a_1=a_2=0$.    
\end{proof}

\begin{remark}
It is observed from the construction that $\tau^{i}$ for $i=1,2$ in Lemma \ref{lem:LBWH} take the form
\begin{align}
\tau^{i}=\left(P^{i}_{j}+\partial_j L^{-1}\partial_k(Q^{kl}P^{i}_{l}) \right)d\theta^j.  \label{eq:tauform}
\end{align}
\end{remark}

\subsection{{Lifted vector fields}}

\begin{proof}[Proof of Theorem \ref{t:basis}]
Let $\tau^{i}$ for $i=1,2$ be the weighted harmonic one-forms on $\Sigma_{\Psi}$ in Lemma \ref{lem:LBWH}. Let $\bol{\tau}^{i}$ be the metric-dual of $\tau^{i}$ for $g_{\Psi}$. Then, by $\bol{\tau}^{i}=\tau^{i\sharp_{g_{\Psi}}}=h^{jm}\tau^{i}_{j}\partial_m$ and \eqref{eq:tauform},
\begin{align}
\bol{\tau}^{i}=\tau^{i\sharp_{g_{\Psi}}}
={h^{jm}\left(P^{i}_{j}+\partial_j L^{-1}\partial_k(Q^{kl}P^{i}_{l}) \right)\partial_m.}  \label{eq:taudual2}
\end{align}
Set the vector field on $\Omega_{\Sigma}$ by 
\begin{align}
\bol{\xi}^{i}(\bol{x})={\bol{\tau}^{i}({\bol{x}})=h^{jm}\left(P^{i}_{j}+\partial_j L^{-1}\partial_k(Q^{kl}P^{i}_{l}) \right)\partial_m.}  \label{eq:xirep}
\end{align}
Then, the weighted co-closedness condition $d\ \tilde{*}\ \tau^{i}=0$ implies that $d\iota_{\bol{\xi}^{i}}dV=0$. Since $\bol{\tau}^{i}$ is tangential to $\Sigma_{\Psi}$, $\iota_{\bol{\xi}^{i}}d\Psi=0$. For an arbitrary tangential vector field $X\in T\Sigma_{\Psi}$ and the one-form $\xi^i=g_{\textrm{Euc}}(\bol{\xi}^i,\cdot)$, we have 
\begin{align*}
{\iota_X i^{*}\xi^{i}=
\iota_{i_{*}X}\xi^{i}
=g_{\textrm{Euc}}(\bol{\xi}^{i},i_{*}X)
=g_{\Psi}(\bol{\tau}^{i},X).}
\end{align*}
Thus $i^{*}\xi^{i}=\tau^{i}$ for $i=1,2$ is the basis of $\tilde{\mathcal{H}}^{1}(\Sigma_{\Psi})$. We proved \eqref{eq:xicondition}.
\end{proof}

\section{Foliation structure}\label{sec:global}

We derive the tangential flow representation \eqref{harmu1} and the normal flux equation \eqref{lnGS}, 
thereby completing the proof of Theorem \ref{thm:lnGS}. 

\subsection{Tangential flow representation}

\begin{lemma}\label{thm:global}
Let $\lr{u,p}$ be a nowhere-vanishing $C^1$ solution of \eqref{SE} in $\Omega_{\Sigma}$. Assume that the conditions \eqref{eq:A1} and \eqref{eq:A2} hold. Then, there exist functions $c_1(\Psi)$ and $c_2(\Psi)$ such that 
\eq{
\bol{u}(\bol{x})
=
c_1(\Psi(\bol{x}))\bol{\xi}^1(\bol{x})+c_2(\Psi(\bol{x}))\bol{\xi}^2(\bol{x}).
\label{gGS0}
}
\end{lemma}

\begin{proof}
By Theorem~\ref{thm:1}, the tangential one-form $v=i^\ast u$ is a weighted harmonic form on $\Sigma_{\Psi}$. We expand $v$ by the basis $\tau^{i}$ as 
\eq{
v=c_1(\Psi)\,\tau^1+c_2(\Psi)\tau^2.\label{v}
}
By the metric-dual tangential vector field $\bol{v}=v^{\sharp_{g_{\Psi}}}=c_1(\Psi)\bol{\tau}^{1}+c_2(\Psi)\bol{\tau}^{2}$, we set 
\begin{align*}
\tilde{\bol{u}}(\bol{x})=\bol{v}(\bol{x})
=c_1(\Psi(\bol{x}))\bol{\xi}^{1}(\bol{x})+c_2(\Psi(\bol{x}))\bol{\xi}^{2}(\bol{x}).
\end{align*}
We set $\bol{w}=\bol{u}-\tilde{\bol{u}}$. The vector field \(\bol{w}\) is tangent to \(\Sigma_\Psi\), because
\(\bol{u}\) and \(\bol{\xi}^i\) are tangent to \(\Sigma_\Psi\) by \eqref{eq:A1} and \eqref{eq:xicondition}. Moreover, it follows from $\tau^{i}=i^{*}\xi^{i}$ that $w=g_{\textrm{Euc}}(\bol{w},\cdot)$ satisfies 
\eq{
i^\ast w
=
i^\ast u
-
c_1(\Psi)i^\ast \xi^1
-
c_2(\Psi)i^\ast \xi^2
=v-c_1(\Psi) \tau^1-
c_2(\Psi) \tau^2
=0.
}
For any
\(X\in T\Sigma_\Psi\), the tangency of \(\bol{w}\) implies 
\eq{
0=\iota_Xi^\ast w
=
\iota_Xw
=
g_{\mathrm{Euc}}(\bol{w},X)
=
g_\Psi(\bol{w},X).
}
In particular, $g_\Psi(\bol{w},\bol{w})=0$. Since the Riemannian metric \(g_\Psi\) is positive
definite on \(\Sigma_\Psi\), \(\bol{w}=0\) and we conclude that \eqref{gGS0} holds. 
\end{proof} 

\begin{prop}
Assume that \eqref{eq:A1} holds. Then, $\bol{u}$ and $u=g_{\textrm{Euc}}(\bol{u},\cdot)$ are expressed as  
\begin{align}
\bol{u}&=u^1(\Psi,\theta^{1},\theta^{2})\partial_{1}
+u^2(\Psi,\theta^{1},\theta^{2})\partial_{2},   \label{eq:tangential2}\\
u&=u_\Psi(\Psi,\theta^{1},\theta^{2})d\Psi+u_1(\Psi,\theta^{1},\theta^{2}) d\theta^1+u_2(\Psi,\theta^{1},\theta^{2}) d\theta^2,
\label{eq:one-form2}
\end{align}
with the components
\begin{align}
u_j(\Psi,\theta^{1},\theta^{2})
&=h_{ij}(\Psi,\theta^{1},\theta^{2})u^i(\Psi,\theta^{1},\theta^{2}),\label{eq:U_B}\\
u_\Psi(\Psi,\theta^{1},\theta^{2})&=g_{\Psi \theta^{i}}(\Psi,\theta^{1},\theta^{2})u^{i}(\Psi,\theta^{1},\theta^{2}).  \label{eq:u_Psi}
\end{align}
\end{prop}

\begin{proof}
This follows from the metric representation \eqref{eq:metric}.
\end{proof}

\begin{prop}
We set 
\begin{align}
U_k=c_i(\Psi)\left(P_j^i+\partial_j L^{-1}\partial_l\left(Q^{lk}P_k^i\right)\right). \label{eq:U}
\end{align}
Then, the velocity field \eqref{gGS0} and its dual one-form are expressed as 
\begin{align}
\bol{u}&=h^{jk}U_k\,\partial_j=h^{1k}U_k\,\partial_1+h^{2k}U_k\,\partial_2,  \label{eq:vtangent}\\
u &=g_{\Psi \theta^i}h^{ij}U_j\,d\Psi+U_1\,d\theta^1+U_2\,d\theta^2.\label{eq:vone-form}
\end{align}
\end{prop}

\begin{proof}
The expression \eqref{eq:vtangent} follows from \eqref{eq:xirep}, \eqref{gGS0}, and \eqref{eq:U}. 
The expression \eqref{eq:vone-form} follows from \eqref{eq:tangential2}-\eqref{eq:one-form2} and \eqref{eq:vtangent}.
\end{proof}

\subsection{The normal flux equation}
\label{sec:lnGS}

We now complete the proof of Theorem \ref{thm:lnGS}. It remains to show the normal flux equation \eqref{lnGS}.

\begin{prop}\label{p:Tchoice}
Let $i:\Sigma_{\Psi}\hookrightarrow\mathbb{R}^3$ denote the inclusion map. For one-forms $u$ in $\Omega_{\Sigma}$ such that $v=i^{*}u$ satisfies $dv=0$, the function $\iota_{T}\iota_{\bol{u}}du$ is independent of the choice of $T$ satisfying \eqref{eq:tran}.
\end{prop}

\begin{proof}
For $T$ and $T'$ satisfying \eqref{eq:tran}, we set $W=T-T'$ and observe that 
\begin{align*}
\iota_{\bol{u}}\iota_{T'}du-\iota_{\bol{u}}\iota_Tdu
=
\iota_{\bol{u}}\iota_Wdu
=
\iota_{\bol{u}}\iota_W\lr{i^*du}
=
\iota_{\bol{u}}\iota_W{dv}
=
0.
\end{align*}
Thus \(\iota_{\bol{u}}\iota_Tdu\) is independent of the tangential part of the transverse lift \(T\). 
\end{proof}

\begin{proof}[Proof of Theorem \ref{thm:lnGS}]
Taking the contraction of the Euler equation \eqref{SE} with $T$, we observe that 
\begin{align*}
\iota_{\bol{u}}\iota_Tdu=p'(\Psi).
\end{align*}
Using \eqref{eq:vone-form}, we compute
\begin{align*}
        du
        =
        \left(
        \partial_\Psi U_j-\partial_j u_\Psi
        \right)
        d\Psi\wedge d\theta^j
        -\partial_k U_j 
        d\theta^j\wedge d\theta^k.
        \end{align*}
By Proposition \ref{p:Tchoice}, we choose $T=\partial_{\Psi}$. Using \eqref{eq:vtangent}, we find that 
\begin{align*}
        \iota_{\bol{u}}\iota_Tdu
        =
        h^{ij}U_j
        \left(
        \partial_\Psi U_i-\partial_i u_\Psi
        \right).
\end{align*}
Substituting \(u_\Psi=g_{\Psi \theta^i}h^{ij}U_j\), we obtain the normal flux equation \eqref{lnGS}. 
\end{proof}

\subsection{The isothermal coodinate form}\label{sec:int}

We show that the normal flux equation \eqref{lnGS} takes a simpler form in the isothermal coordinates adapted to the foliation.

\begin{lemma}\label{l:isonlGS}
In the adapted isothermal coordinates $(\Psi,\chi,\phi)$ in Proposition \ref{p:iso1}, the normal flux equation \eqref{lnGS} is written as 
\begin{align}
-\mf{L}_{\bol{u}} u_{\Psi}
+
\displaystyle\frac{\abs{\bol{u}}}{\lambda}
\frac{\p\lr{\lambda\abs{\bol{u}}}}{\p\Psi}
=
p'(\Psi),  \label{eq:lnGSiso}
\end{align}
where $\mf{L}_{\bol{u}}$ denotes the Lie derivative along with $\bol{u}$,  $u_{\Psi}$ denotes the covariant $\Psi$-component of $\bol{u}$, and $\lambda$ denotes the conformal factor.
\end{lemma}

\begin{proof}
The normal flux equation \eqref{lnGS} is equivalent to the first equation of \eqref{eq:SEisoflux0}. Namely, $u=u_{\Psi}d\Psi+u_{\chi}d\chi+u_{\phi}d\phi$ and $\bol{u}=u^{\chi}\partial_\chi+u^{\phi}\partial_\phi$ satisfy 
\begin{equation}
\begin{aligned}
u^\chi
\left(
\partial_\Psi u_\chi-\partial_\chi u_\Psi
\right)
+
u^\phi
\left(
\partial_\Psi u_\phi-\partial_\phi u_\Psi
\right)
=
\partial_\Psi p.
\end{aligned}
\label{eq:SEisoflux}
\end{equation}
The left-hand side is expressed as 
\begin{equation}
-\mf{L}_{\bol{u}}u_{\Psi}+{u^{\chi}\partial _\Psi u_{\chi}+u^{\phi}\partial _\Psi u_{\phi}}.
\end{equation}
Since $(\chi,\phi)$ is isothermal, we have $u_{\chi}=\lambda^{2}u^{\chi}$ and $u_{\phi}=\lambda^{2}u^{\phi}$. This yields  
\begin{align*}
|\bol{u}|^{2}=\iota_{\bol{u}}{u}
=\lambda^{2}(|u^{\chi}|^{2}+|u^{\phi}|^{2})=\frac{1}{\lambda^{2}}(|u_{\chi}|^{2}+|u_{\phi}|^{2}).
\end{align*}
It follows that 
\begin{align*}
{u^{\chi}\partial _\Psi u_{\chi}+u^{\phi}\partial _\Psi u_{\phi}}
=\frac{1}{\lambda^{2}}(u_{\chi}\partial _\Psi u_{\chi}+u_{\phi}\partial _\Psi u_{\phi}) 
=\frac{1}{2\lambda^{2}}\partial_{\Psi}(\lambda^{2}|\bol{u}|^{2})
={\abs{\bol{u}}^2\partial_{\Psi}\log{\lambda}+|\bol{u}|\partial_{\Psi}|\bol{u}|}
=\frac{|\bol{u}|}{\lambda}\partial_{\Psi}(\lambda |\bol{u}|).
\end{align*}
Thus \eqref{eq:lnGSiso} holds.
\end{proof}

\subsection{The axisymmetric case}

We deduce the Clebsch representation \eqref{Clebsch} and the Grad--Shafranov equation \eqref{GSeq0} from Theorem \ref{thm:lnGS} in the axisymmetric case.

\begin{prop}
Let $\Omega_{\Sigma}$ be an axisymmetric toroidal domain, and let
$\Psi=\Psi(r,z)$ be an axisymmetric flux function in cylindrical
coordinates $(r,\phi,z)$. 
{Let \(\bol{x}_0\in\Omega_{\Sigma}\), and set
\(\Psi_0=\Psi(\bol{x}_0)\). After an axisymmetric rotation, we may assume 
that \(\bol{x}_0\in\{\phi=0\}\).}
{
Let
\[
(r(\Psi,s),z(\Psi,s))
\]
be a smooth local parametrization of the meridional level curves
\(\Sigma_{\Psi}|_{\phi=0}\) for \(\Psi\) near \(\Psi_0\), where \(s\)
is the arc-length parameter on each curve. Define \(\chi\) locally by
\eq{
\frac{\partial\chi}{\partial s}=\frac{1}{r(\Psi,s)},
\qquad
\chi(\Psi,0)=0.\label{eq:chis}
}
Then \((\Psi,\chi,\phi)\) form adapted isothermal coordinates on a
neighborhood of \(\bol{x}_0\) satisfying 
}
\begin{align}
&\lambda=r,\quad J=\frac{r^{2}}{|\nabla \Psi|}, \label{eq:iso1}\\
&\partial_{\Psi}=\partial_\Psi r \partial_r+\partial_\Psi z \partial_z,\quad \partial_{\chi}=\frac{r}{|\nabla \Psi|}(\partial_z\Psi\partial_r-\partial_r\Psi\partial_z), \label{eq:iso2}\\
&d\Psi=\partial_r\Psi d r+\partial_z\Psi d z,\quad d\chi=\frac{|\nabla \Psi|}{r}(\partial_\Psi z dr-\partial_\Psi r dz), \label{eq:iso3}\\
&d \Psi\wedge d \chi=\frac{|\nabla \Psi|}{r}dz\wedge dr. \label{eq:iso4}
\end{align}
\end{prop}

\begin{proof}
{Omitting the leaf label, we parametrize the torus $\Sigma_{\Psi}$} by 
\begin{align*}
X(s,\phi)=(r(s)\cos\phi,r(s)\sin\phi,z(s)).
\end{align*}
Using the tangential derivatives 
\begin{align*}
\partial_s X=(r'(s)\cos\phi,r'(s)\sin\phi,{z'}(s)),\quad 
\partial_\phi X=(-r(s)\sin\phi,r(s)\cos\phi,0),
\end{align*}
we compute the $(s,\phi)$-coordinates metric expression $g_{\Psi}=ds^{2}+r^{2}d\phi^{2}$. Since $ds^{2}=r^{2}d\chi^{2}$, $(\chi,\phi)$ are the isothermal coordinates with $\lambda=r$. The identity $J=r^{2}/|\nabla \Psi|$ follows from $dV=Jd\Psi\wedge d\chi\wedge d\phi=rdr\wedge d\phi\wedge dz$ and \eqref{eq:iso4}{, which is shown below}.

Since $\partial_s$ has length one and it is tangent to the curve {$\Sigma_{\Psi}|_{\phi=0}$, upon choosing an orientation,} we find that 
\begin{align*}
\partial_s=\frac{1}{|\nabla \Psi|}\left(\partial_z \Psi \partial_r-\partial_r \Psi \partial_z\right).
\end{align*}
{Combining this expression with} \eqref{eq:chis}, we obtain \eqref{eq:iso2}, 
\begin{align}
\partial_{\chi}
=\frac{r}{|\nabla \Psi|}\left(\partial_z \Psi {\p_r}-\partial_r \Psi {\p_z}\right).  \label{eq:dchi}
\end{align}
{Setting $d\chi=Adr+Bdz$,} 
it follows from $\iota_{\partial_\Psi}d\chi=0$ and $\iota_{\partial_\chi}d\chi=1$ that 
\begin{align*}
\begin{pmatrix}
r_{\Psi} & z_{\Psi}\\
r_{\chi} & z_{\chi}
\end{pmatrix}
\begin{pmatrix}
A \\
B
\end{pmatrix}
=
\begin{pmatrix}
0 \\
1
\end{pmatrix},
\end{align*}
{where we used the notation $r_{\Psi}=\p r/\p\Psi$.} 
Differentiating $\Psi(r(\Psi,\chi),z(\Psi,\chi))=\Psi$ in $\Psi$ yields 
\begin{align}
r_{\Psi}\partial_r \Psi+ z_{\Psi}\partial_z \Psi=1.  \label{eq:identityPsi}
\end{align}
This identity and \eqref{eq:dchi} imply 
\begin{align*}
{{\frac{\partial(r,z)}{\partial(\Psi,\chi)}}}=r_{\Psi}z_{\chi}-r_{\chi}z_{\Psi}=-\frac{r}{|\nabla \Psi|}.
\end{align*}
We thus obtain the coefficients
\begin{align*}
\begin{pmatrix}
A \\
B
\end{pmatrix}
=\left({{\frac{\partial(r,z)}{\partial(\Psi,\chi)}}}\right)^{-1}
\begin{pmatrix}
z_{\chi} & -z_{\Psi}\\
-r_{\chi} & r_{\Psi}
\end{pmatrix}
\begin{pmatrix}
0 \\
1
\end{pmatrix}
=\frac{|\nabla \Psi|}{r}
\begin{pmatrix}
z_{\Psi} \\
-r_{\Psi}
\end{pmatrix},
\end{align*}
and the expression \eqref{eq:iso3}. The identity \eqref{eq:iso4} follows from \eqref{eq:iso3} and \eqref{eq:identityPsi}.
\end{proof}

\begin{prop}\label{p:xichoice}
The vector fields \eqref{eq:xibasis} and thier metric-duals $\xi^{i}=g_{\textrm{Euc}}(\bol{\xi}^{i},\cdot)$ are expressed as 
\begin{align}
\bol{\xi}^{1}&=-\frac{|\nabla \Psi|}{r^{2}}\partial_{\chi},\quad \bol{\xi}^{2}=\frac{1}{r^{2}}\partial_{\phi},  \label{eq:exsp1}\\
\xi^{1}&=\xi^{1}_{\Psi}d\Psi-|\nabla \Psi|d\chi,\quad \xi^{2}=
d\phi. \label{eq:exsp2}
\end{align}
The vector fields $\bol{\xi}^{i}$ are lifted weighted harmonic vector fields satisfying \eqref{eq:xicondition}.
\end{prop}

\begin{proof}
The expression \eqref{eq:exsp1} follows from \eqref{eq:iso2}. The expression \eqref{eq:exsp2} follows from the identity \eqref{eq:one-form} and \eqref{eq:exsp1}. The weighted harmonicity of $\tau^{i}=i^{*}\xi^{i}$ follows from \eqref{eq:divfreeiso}, \eqref{eq:iso1}, and \eqref{eq:exsp2}.
\end{proof}

\begin{prop}\label{p:expressions}
Let $\bol{u}$ satisfy the Clebsch representation \eqref{Clebsch}. Then, the following expressions hold:
\begin{align}
\bol{u}&=-\frac{|\nabla \Psi|}{r^{2}}\partial_{\chi}+\frac{\alpha(\Psi)}{r^{2}}\partial_{\phi},  \label{eq:ident1}\\
u&=u_{\Psi}d\Psi-|\nabla \Psi|d\chi+\alpha(\Psi)d\phi, \label{eq:ident2} \\
\Delta^{\ast}\Psi  
&=r^{2} u^{\chi}\left(\p_\Psi u_{\chi} - \p_\chi u_{\Psi}\right). \label{eq:ident3}
\end{align}
\end{prop}

\begin{proof}
Using $\nabla \phi=r^{-2}\partial_{\phi}$, it follows from \eqref{eq:iso2} that 
\begin{align*}
\nabla \Psi\times \nabla \phi=(\partial_r\Psi\partial_r+\partial_z\Psi\partial_z)\times \frac{1}{r^{2}}\partial_{\phi}
=-\frac{1}{r}(\partial_z\Psi\partial_r-\partial_r\Psi\partial_z)=-\frac{|\nabla\Psi|}{r^{2}}\partial_{\chi}.
\end{align*}
Thus the expression \eqref{eq:ident1} holds. Since $(\chi,\phi)$ is isothermal, the expression \eqref{eq:ident2} follows from \eqref{eq:one-form}. Differentiating $u$, we apply \eqref{eq:iso4} and \eqref{eq:ident2} to obtain 
\begin{align*}
du=(\partial_{\Psi}u_{\chi}-\partial_{\chi}u_{\Psi})d\Psi\wedge d \chi+\alpha'(\Psi)d\Psi\wedge d\phi
=-\frac{{u_{\chi}}}{r}(\partial_{\Psi}u_{\chi}-\partial_{\chi}u_{\Psi})dz\wedge d r+\alpha'(\Psi)d\Psi\wedge d\phi.
\end{align*}
Differentiating $u=-r^{-1}\partial_z\Psi dr+r^{-1}\partial_z\Psi dz+\alpha(\Psi)d\phi$ in turn yields
\begin{align*}
du=-\frac{1}{r}\Delta^{*}\Psi dz\wedge dr+\alpha'(\Psi)d\Psi\wedge d\phi.
\end{align*}
Thus the last expression \eqref{eq:ident3} follows.
\end{proof}

\begin{proof}[Proof of Theorem \ref{t:axisymmetric}]
The vector fields \eqref{eq:xibasis} satisfy the properties \eqref{eq:xicondition} by Proposition \ref{p:xichoice}. Thus the tangential flow representation \eqref{harmu1} holds. Namely,
\begin{equation*}
\bol{u}=c_1(\Psi)\nabla\Psi\cp\nabla\phi+c_2(\Psi)\nabla\phi, 
\end{equation*}
with some functions $c_i(\Psi)$. If $c_1(\Psi)=1$ and $c_2(\Psi)=\alpha(\Psi)$, the Clebsch representation \eqref{Clebsch} holds. The normal flux equation \eqref{eq:lnGSiso} is identical to 
\begin{equation*}
\begin{aligned}
u^\chi
\left(
\partial_\Psi u_\chi-\partial_\chi u_\Psi
\right)
+
u^\phi
\left(
\partial_\Psi u_\phi-\partial_\phi u_\Psi
\right)
=
\partial_\Psi p.
\end{aligned}
\end{equation*} 
The first and second terms agree with $r^{-2}\Delta^{*}\Psi$ and $r^{-2}\alpha(\Psi)\alpha'(\Psi)$ by the identity \eqref{eq:ident3} and the expressions \eqref{eq:ident2}-\eqref{eq:ident3}. Thus $\Psi$ satisfies the Grad-Shafranov equation \eqref{GSeq0}.
\end{proof}

\if{

\begin{remark}[The flux computation]\label{rem:iteration}
The normal flux equation \eqref{lnGS} does not appear to be elliptic in general unlike the classical equation \eqref{GSeq0}. One may try to compute a solution to \eqref{lnGS} for prescribed functions $c_i(\Psi)$ and $p(\Psi)$ by repeating the following procedures (i)-(v):
\begin{itemize}
\item[(i)] For a flux function $\Psi$, compute the leafwise data \eqref{eq:MM}-\eqref{eq:Ell} 
\item[(ii)] Compute the lifted solenoidal vector fields $\bol{\xi}^{i}$ by \eqref{eq:xirep}
\item[(iii)] Evaluate the scalar residual (LHS)-(RHS) in \eqref{lnGS}
\item[(iv)] Update the flux function to $\Psi'=\Psi+\delta\Psi$ by choosing $\delta\Psi$ so as to reduce the scalar residual
\end{itemize}
If an invariant-torus foliation exists, and if the chosen update scheme is well posed, one expects that this iteration converges toward a solution to \eqref{lnGS}.
\end{remark}

}\fi

\if{

\appendix

\section{Local structure of steady Euler flows foliated by  invariant tori  }\label{sec:local}\label{sec:local}

In this section, we establish a local structure
theorem for steady Euler flows foliated by invariant tori with the aid of the adapted isothermal coordinates constructed in subsec.~\ref{subsec:iso}:

\begin{thm}[Local structure and geometric constraints]\label{thm:structure}
Under the assumptions of Theorem~\ref{thm:harmonic}, {write in 
a local adapted chart $\lr{\Psi,\chi,\phi}$, where $\lr{\chi,\phi}$ are local isothermal coordinates on the leaf $\Sigma_{\Psi}$,}   
\eq{
\bol{u}=u^{\chi}\p_\chi+u^{\phi}\p_\phi,\qquad u=\bol{u}^{\flat}.
}
Let $\bol{x}\in\Omega_{\Sigma}$ be a point at which $u^{\chi}$ and $u^{\phi}$ do not vanish simultaneously.
Then there exists an open neighborhood $U\ni \bol{x}$ such that  the following holds on $U$.

\begin{itemize}
\item[(i)] \textbf{Local structure.}
Locally in a sufficiently small neighborhood  $V\subseteq U$
there exists a potential $\sigma\lr{\Psi,\chi,\phi}$ such that 
\eq{
u=d\sigma-\lrs{\frac{\p\sigma}{\p\Psi}-\frac{1}{\lambda^2}\lr{g_{\Psi\chi}\frac{\p\sigma}{\p\chi}+g_{\Psi\phi}\frac{\p\sigma}{\p\phi}}}d\Psi,
}
where $\lambda^2=g_{\chi\chi}=g_{\phi\phi}$, $g_{\Psi\chi}$ and $g_{\Psi\phi}$ are the covariant metric coefficients. 
\item[(ii)] \textbf{Geometric constraints.} 
In terms of $\sigma$,
denoting by $J$ the Jacobian of the isothermal coordinate transformation, the reduced Euler system is equivalent to
\sys{
&\frac{1}{2}\frac{\p}{\p\Psi}\lr{\sigma_{\chi}^2+\sigma_{\phi}^2}-\lr{\sigma_{\chi}\frac{\p}{\p\chi}+\sigma_{\phi}\frac{\p}{\p\phi}}\lr{\frac{g_{\Psi\chi}\sigma_{\chi}+g_{\Psi\phi}\sigma_{\phi}}{\lambda^2}}=\lambda^2p',\label{traC}\\
&\frac{\p}{\p\chi}\lr{\frac{J}{\lambda^2}\sigma_{\chi}}
+\frac{\p}{\p\phi}\lr{\frac{J}{\lambda^2}\sigma_{\phi}}
=0,\label{ellC}
}{SER3ii2_main}
where $\sigma_{\chi}=\p\sigma/\p\chi$ and $\sigma_{\phi}=\p\sigma/\p\phi$. 
{Equation \eqref{ellC} represents a weighted elliptic constraint and equation \eqref{traC} a scalar compatibility condition evaluated on its solutions.}  
\end{itemize}
\end{thm}

We now prove Theorem~\ref{thm:structure} for 
a general foliation geometry with the aid of the isothermal coordinates constructed in subsec.~\ref{subsec:iso}. 

 \begin{proof}
 (i) Because $u\neq 0$, the second and third equations of the reduced Euler system \eqref{SERiso} imply that in a sufficiently small neighborhood $V\subseteq U\subseteq\Omega_{\Sigma}$ we can define a potential $\sigma\lr{\Psi,\chi,\phi}$ such that
\eq{
u_{\chi}=\frac{\p\sigma}{\p\chi},\qquad u_{\phi}=\frac{\p\sigma}{\p\phi}.
}
A possible choice in a neighborhood $V$ of a point $\lr{\Psi,\chi_0,\phi_0}$  is given by
\eq{
\sigma=\varrho\lr{\Psi,\phi}+\int_{\chi_0}^{\chi}u_{\chi}\,d\chi,\qquad\frac{\p\varrho}{\p\phi}=u_{\phi}\lr{\Psi,\chi_0,\phi}. 
}
On the other hand, the covariant components of the velocity $1$-form $u$ expressed in the isothermal coordinates are given by \eqref{contra}, so that, in $V$, 
\eq{
u=&\frac{1}{\lambda^2}\lr{g_{\Psi\chi}\frac{\p\sigma}{\p\chi}+g_{\Psi\phi}\frac{\p\sigma}{\p\phi}}\,d\Psi+\frac{\p\sigma}{\p\chi}\,d\chi+\frac{\p\sigma}{\p\phi}\,d\phi
=d\sigma-\lrs{\frac{\p\sigma}{\p\Psi}-\frac{1}{\lambda^2}\lr{g_{\Psi\chi}\frac{\p\sigma}{\p\chi}+g_{\Psi\phi}\frac{\p\sigma}{\p\phi}}}d\Psi.\label{uiso}
}
(ii) In $V$, substituting the expression \eqref{uiso} of $u$ into the governing equations~\eqref{SERiso}, 
we further obtain the reduced Euler system
\sys{
&\frac{1}{2}\frac{\p}{\p\Psi}\lr{\sigma_{\chi}^2+\sigma_{\phi}^2}+\lr{\sigma_{\chi}\frac{\p}{\p\chi}+\sigma_{\phi}\frac{\p}{\p\phi}}\lr{
\frac{g^{\Psi\chi}\sigma_{\chi}+g^{\Psi\phi}\sigma_{\phi}}{g^{\Psi\Psi}}}=\lambda^2p',\\
&\frac{\p}{\p\chi}
\lr{
\frac{\sigma_{\chi}}{\sqrt{g^{\Psi\Psi}}}
}
+\frac{\p}{\p\phi}\lr{
\frac{\sigma_{\phi}}{\sqrt{g^{\Psi\Psi}}}
}
=0,\label{ellX}
}{SERisoX}
where $\sigma_{\chi}=\p\sigma/\p\chi$
and $\sigma_{\phi}=\p\sigma/\p\phi$. 
Note that system \eqref{SERisoX} is equivalent to \eqref{SER3ii2_main} upon substitution of \eqref{covg}.
The second equation \eqref{ellX} is a second-order elliptic equation,
\eq{
\sigma=\tilde{\sigma}\lrs{\frac{J}{\lambda^2}}=\tilde{\sigma}\lrs{g^{\Psi\Psi}},
}
which determines $\sigma$ on each level surface $s_{\Psi}=\Sigma_{\Psi}\cap V$ 
under suitable boundary conditions on $\p s_{\Psi}$. 
The remaining equation represents a geometric constraint on the shape of the level surface, 
\eq{
\frac{1}{2}\frac{\p}{\p\Psi}\lr{\tilde{\sigma}_{\chi}^2+\tilde{\sigma}_{\phi}^2}
+\lr{\tilde{\sigma}_{\chi}\frac{\p}{\p\chi}+\tilde{\sigma}_{\phi}\frac{\p}{\p\phi}}\lr{\frac{g^{\Psi\chi}\tilde{\sigma}_{\chi}+g^{\Psi\phi}\tilde{\sigma}}{g^{\Psi\Psi}}}=\lambda^2p'.
\label{gcX}
}
\end{proof}

\begin{remark}[Interpretation of Theorem~\ref{thm:structure}]\label{rem:int}
Theorem~\ref{thm:structure} shows that, given a hollow toroidal domain $\Omega_{\Sigma}$ foliated by nested flux surfaces $\Sigma_{\Psi}$, if the geometric constraint \eqref{gcX} is satisfied for some potential $\tilde{\sigma}$ obtained as solution of \eqref{ellX}, then steady solutions of the Euler equations with invariant tori exist in the local form \eqref{uiso}. Moreover, by Theorem~\ref{thm:harmonic}, their tangential restrictions are weighted harmonic $1$-forms on each invariant torus. Thus, the existence problem for steady Euler flows with invariant tori reduces to a geometric compatibility condition on the foliation itself.

More precisely, this amounts to the possibility of finding a foliation whose conformal factor $\lambda^2\lr{\Psi,\chi,\phi}$, which determines the intrinsic curvature of each leaf, satisfies \eqref{gcX}. Indeed, from $g^{\Psi\Psi}=\lambda^4/J^2$, $g^{\Psi\chi}=-\lambda^2 g_{\Psi\chi}/J^2$, $g^{\Psi\phi}=-\lambda^2 g_{\Psi\phi}/J^2$, and the Jacobian identity
\eq{
J=\lambda\sqrt{\lambda^2 g_{\Psi\Psi}-g_{\Psi\phi}^2-g_{\Psi\chi}^2},
}
one obtains
\eq{
g_{\Psi\Psi}=
\frac{1}{g^{\Psi\Psi}}
\lrc{
1+\frac{\lambda^2}{g^{\Psi\Psi}}
\lrs{
\lr{g^{\Psi\chi}}^2+\lr{g^{\Psi\phi}}^2
}
}.
}
Hence the covariant metric coefficients $g_{\Psi\Psi}$, $g_{\Psi\chi}$, $g_{\Psi\phi}$, $\lambda^2$,  can be expressed entirely in terms of $g^{\Psi\Psi}$, $g^{\Psi\chi}$, $g^{\Psi\phi}$, $\lambda^2$. Therefore, a necessary and sufficient condition for the existence of the present type of steady Euler flow in $V$ is that there exist a local potential $\tilde{\sigma}$, solution of \eqref{ellX}, such that the conformal factor $\lambda^2$ appearing within the Euclidean line element  
\eq{
ds^2=
\frac{1+\frac{\lambda^2}{g^{\Psi\Psi}}\lrs{\lr{g^{\Psi\chi}}^2+\lr{g^{\Psi\phi}}^2}}{g^{\Psi\Psi}}\,d\Psi^2
-2\frac{\lambda^2\lr{g^{\Psi\chi}d\chi+g^{\Psi\phi}d\phi}}{g^{\Psi\Psi}}\,d\Psi
+\lambda^2\lr{d\chi^2+d\phi^2},
}
satisfies
\eq{
\lambda^2=
\frac{1}{p'}
\lrs{
\frac{1}{2}\frac{\p}{\p\Psi}\lr{\tilde{\sigma}_{\chi}^2+\tilde{\sigma}_{\phi}^2}
+
\lr{\tilde{\sigma}_{\chi}\frac{\p}{\p\chi}+\tilde{\sigma}_{\phi}\frac{\p}{\p\phi}}
\lr{
\frac{g^{\Psi\chi}\tilde{\sigma}_{\chi}+g^{\Psi\phi}\tilde{\sigma}_{\phi}}{g^{\Psi\Psi}}
}
}.\label{lambda2} 
}

On each leaf $\Sigma_{\Psi}$, the induced metric is
\eq{
g_{\Psi}=\lambda^2\lr{d\chi^2+d\phi^2}=e^{2f}\lr{d\chi^2+d\phi^2},
\qquad f=\log\lambda.
}
Hence, by the two-dimensional conformal curvature formula \cite[thm.~7.30]{Lee2018}, the scalar curvature
$R=g^{ij}R_{ij}$ of the leaf, where $R_{ij}$ is the Ricci curvature tensor, satisfies
\eq{
R=-\frac{1}{\lambda^2}\lr{\frac{\p^2}{\p\chi^2}+\frac{\p^2}{\p\phi^2}}\log\lambda^2.
}
Equivalently, since $R=2K$ on a surface, the Gaussian curvature is
\eq{
K=-\frac{1}{\lambda^2}\lr{\frac{\p^2}{\p\chi^2}+\frac{\p^2}{\p\phi^2}}\log\lambda.
}
Thus the geometric condition \eqref{lambda2} may be viewed as a constraint on the intrinsic curvature of the leaves, encoded through the conformal factor $\lambda^2$.
\end{remark}

}\fi

\section*{Statements and Declarations}
\subsection*{Data availability}
Data sharing not applicable to this article as no datasets were generated or analysed during the current study.


\subsection*{Competing interests} 
The authors have no competing interests to declare that are relevant to the content of this article.



\bibliographystyle{alphaurl}
\bibliography{refs_initials}



 



\end{document}